\documentclass{amsart}
\usepackage{latexsym}
\usepackage{amsfonts}
\usepackage{amssymb}
\usepackage{mathrsfs}
\usepackage[all]{xy}
\usepackage{bbm}

\addtocounter{section}{-1}

\newtheorem{theorem}{Theorem}[section]
\newtheorem{corollary}[theorem]{Corollary}
\newtheorem{lemma}[theorem]{Lemma}
\newtheorem{proposition}[theorem]{Proposition}

\theoremstyle{definition}
\newtheorem{definition}{Definition}[section]

\theoremstyle{remark}
\newtheorem{remark}{Remark}[section]
\numberwithin{equation}{section}

\newcommand{\CC}{\mathbb{C}}
\newcommand{\RR}{\mathbb{R}}

\newcommand{\QQ}{\mathbb{Q}}

\newcommand{\KK}{\mathbbm{k}}

\newcommand{\Lk}{\mathfrak{k}}

\newcommand{\Lgl}{\mathfrak{gl}}
\newcommand{\Lf}{\mathfrak{f}}
\newcommand{\Lg}{\mathfrak{g}}
\newcommand{\Lh}{\mathfrak{h}}

\newcommand{\La}{\mathfrak{a}}
\newcommand{\Lb}{\mathfrak{b}}
\newcommand{\Lc}{\mathfrak{c}}
\newcommand{\Lz}{\mathfrak{z}}

\newcommand{\A}{\mathrm{A}}
\newcommand{\Diff}{\mathrm{Diff}}

\newcommand{\Iso}{\mathrm{Iso}}
\newcommand{\germ}{\mathrm{germ}}
\newcommand{\Kill}{\mathrm{Kill}}
\newcommand{\Vect}{\mathrm{Vect}}

\newcommand{\Hull}{\mathrm{Hull}}

\newcommand{\ev}{\mathrm{ev}}
\newcommand{\Ad}{\mathrm{Ad}}
\newcommand{\ad}{\mathrm{ad}}
\newcommand{\id}{\mathrm{id}}
\newcommand{\loc}{\mathrm{loc}}

\newcommand{\reg}{\mathrm{reg}}

\newcommand{\GL}{\mathrm{GL}}
\newcommand{\SO}{\mathrm{SO}}

\newcommand{\mt}{\mapsto}
\newcommand{\LL}{\mathscr{L}}

\newcommand{\sA}{\mathscr{A}}

\newcommand{\bG}{\mathbf{G}}
\newcommand{\bH}{\mathbf{H}}
\newcommand{\bN}{\mathbf{N}}
\newcommand{\bK}{\mathbf{K}}
\newcommand{\bS}{\mathbf{S}}
\newcommand{\bV}{\mathbf{V}}

\newcommand{\cV}{\mathcal{V}}
\newcommand{\cW}{\mathcal{W}}
\newcommand{\cG}{\mathcal{G}}

\newcommand{\cZ}{\mathcal{Z}}

\begin{document}

\title[Rigid geometric structures]
{Rigid geometric structures, isometric actions, and algebraic
quotients}

\author[J. An]{Jinpeng An}
\address{LMAM, School of Mathematical Sciences, Peking University, Beijing, 100871, China}
\email{anjinpeng@gmail.com}

%\subjclass[2000]{53C15, 57S20, 37B05, 37C85, 20G20, 22D40.}
%
%\keywords{Rigid geometric structure, isometric action, invariant
%measure, fundamental group, algebraic quotient.}

\thanks{Research partially supported by NSFC grant 10901005 and FANEDD grant 200915.}

\begin{abstract}
By using a Borel density theorem for algebraic quotients, we prove a
theorem concerning isometric actions of a Lie group $G$ on a smooth
or analytic manifold $M$ with a rigid $\A$-structure $\sigma$. It
generalizes Gromov's centralizer and representation theorems to the
case where $R(G)$ is split solvable and $G/R(G)$ has no compact
factors, strengthens a special case of Gromov's open dense orbit
theorem, and implies that for smooth $M$ and simple $G$, if Gromov's
representation theorem does not hold, then the local Killing fields
on $\widetilde{M}$ are highly non-extendable. As applications of the
generalized centralizer and representation theorems, we prove (1) a
structural property of $\Iso(M)$ for simply connected compact
analytic $M$ with unimodular $\sigma$, (2) three results
illustrating the phenomenon that if $G$ is split solvable and large
then $\pi_1(M)$ is also large, and (3) two fixed point theorems for
split solvable $G$ and compact analytic $M$ with non-unimodular
$\sigma$.
\end{abstract}

\maketitle

\section{Notation and conventions}\label{S:notation}

Throughout this paper, $M$ denotes a connected $C^\varepsilon$
manifold, where $\varepsilon=\infty$ or $\omega$ (here $C^\omega$
means real analytic). In most cases, $M$ is endowed with a rigid
$C^\varepsilon$ geometric structure of algebraic type (abbreviated
as $\A$-structure) in the sense of Gromov \cite{Gr}. We denote the
universal cover of $M$ by $\widetilde{M}$, with covering map
$\pi:\widetilde{M}\to M$. For simplicity, sometimes we also denote
$\Gamma=\pi_1(M)$. If $f$ is a local $C^\varepsilon$ diffeomorphism
of $M$ defined around $x\in M$, we denote the germ of $f$ at $x$ by
$\langle f\rangle_x$. Similarly, if $v$ is a $C^\varepsilon$ local
vector field defined around $x$, we denote the germ of $v$ at $x$ by
$\langle v\rangle_x$.

Suppose that $M$ is endowed with a $C^\varepsilon$ geometric
structure $\sigma$. We denote by $\Iso(M)$ the $C^\varepsilon$
isometry group of $M$, by $\Iso^\germ(M)$ the groupoid of germs of
local $C^\varepsilon$ isometries of $M$, by $\Kill(M)$ the Lie
algebra of $C^\varepsilon$ Killing fields on $M$, and by
$\Kill^\germ_x(M)$ the Lie algebra of germs at $x$ of local
$C^\varepsilon$ Killing fields defined around $x$. We also denote
the evaluation map at $x$ from $\Kill(M)$ or $\Kill^\germ_x(M)$ to
$T_xM$ by the same symbol $\ev_x$. Note that if $\sigma$ is rigid,
then $\Iso(M)$ is a Lie group, $\Kill(M)$ and $\Kill^\germ_x(M)$ are
finite-dimensional.

\begin{remark}
We prefer to work with the groupoid $\Iso^\germ(M)$ rather than the
more commonly used pseudogroup $\Iso^\loc(M)$ of local isometries.
This makes the notion of Zariski hulls, which will be introduced
later, easier to define and use. The orbits of $\Iso^\germ(M)$ and
$\Iso^\loc(M)$ are the same. For basic facts concerning groupoids,
the reader may refer to \cite{MM}.
\end{remark}

We denote by $\LL$ the operation which takes a Lie group to its Lie
algebra. If a Lie group is denoted by a capital Latin letter, we
also denote its Lie algebra by the corresponding small Gothic
letter. If a group $G$ acts on a set $X$, we denote the set of
$G$-orbits in $X$ by $X/G$, denote the set of $G$-fixed points in
$X$ by $X^G$, and denote the stabilizer of $x\in X$ in $G$ by
$G(x)$. If $G$ is a Lie group and $G(x)$ is closed, we denote
$\Lg(x)=\LL(G(x))$.

Whenever we say a Lie group $G$ acts on $M$ by $C^\varepsilon$
isometries, we always assume that the action map $G\times M\to M$ is
$C^\varepsilon$. In this case, for every $x\in M$, the surjective
linear map $\iota_x:\Lg\to T_xGx$ defined by
$\iota_x(X)=\left.\frac{d}{dt}\right|_{t=0}\exp(-tX)x$ has kernel
$\Lg(x)$. We refer to the map $\iota:\Lg\to\Kill(M)$,
$\ev_x(\iota(X))=\iota_x(X)$ as the induced infinitesimal action of
$\Lg$ on $M$. It is well-known that $\iota$ is a Lie algebra
homomorphism and is $G$-equivariant with respect to the adjoint
representation and the natural representation $g\mt dg$ of $G$ in
$\Kill(M)$. We refer to $\cG=\iota(\Lg)$ as the Lie algebra of
Killing fields induced by the $G$-action. If $G$ is connected with
universal cover $\widetilde{G}$, then the $G$-action induces a
$C^\varepsilon$ isometric action of $\widetilde{G}$ on
$\widetilde{M}$. We denote by $\widetilde{\cG}$ the Lie algebra of
Killing fields induced by the $\widetilde{G}$-action on
$\widetilde{M}$, which is obviously equal to the image of $\cG$
under the natural injection $\Kill(M)\to\Kill(\widetilde{M})$.

Let $\KK$ be a locally compact non-discrete field. By a
$\KK$-variety (resp. $\KK$-group) we mean an algebraic variety
(resp. linear algebraic group) defined over $\KK$. We use a
subscript ``$0$" to denote the identity component of a $\KK$-group
(resp. Lie group) under the Zariski (resp. Hausdorff) topology. The
set (resp. group) of $\KK$-rational points of a $\KK$-variety (resp.
$\KK$-group) is indicated by a subscript ``$\KK$". Unless otherwise
specified, whenever a topological statement concerning such a set
(resp. group) is made, the topology is understood as the Hausdorff
one.

We refer to the set (resp. group) of $\RR$-rational points of an
$\RR$-variety (resp. $\RR$-group) as a real algebraic variety (resp.
group). A real algebraic group can be also recognized as an
algebraic subgroup of $\GL(\cV)$, where $\cV$ is some
finite-dimensional real vector space. If $G$ is a subgroup of
$\GL(\cV)$, we denote by $\overline{\overline{G}}$ the Zariski
closure of $G$ in $\GL(\cV)$. Note that if $G$ is a real algebraic
group, $G_0$ refers to the identity component of $G$ under the
Hausdorff topology.

Let $G$ be a Lie group, and let $\rho$ be a finite-dimensional real
representation of $G$. Motivated by \cite{Sh} (and \cite{BFM,Me}),
we say that $G$ is \emph{discompact with respect to $\rho$} (or
\emph{$\rho$-discompact} for short) if
$\overline{\overline{\rho(G)}}$ has no proper normal cocompact
algebraic subgroups. We make the convention that $G$ is
$\rho$-discompact if the representation space is $0$. Following
\cite{Kn}, if $G$ is connected and solvable, we say that $G$ is
\emph{split solvable} if all eigenvalues of $\Ad(g)$ are real for
every $g\in G$.

Let $\La$ be a (real) Lie algebra. We write $\La_1<\La$ (resp.
$\La_1\lhd\La$) to indicate that $\La_1$ is a subalgebra (resp.
ideal) of $\La$. Recall that a Lie algebra $\Lb$ is called a
subquotient of $\La$ if there exist $\La_1<\La$ and $\La_2\lhd\La_1$
such that $\Lb\cong\La_1/\La_2$. We denote this relation by
$\Lb\prec\La$. From the Levi decomposition, it is easy to see that
if $\Lb$ is semisimple and $\Lb\prec\La$, then $\La$ has a
subalgebra isomorphic to $\Lb$.

\section{Introduction}\label{S:introduction}

\subsection{Background}

Motivated by the work of Zimmer \cite{Zi86} on rigidity of group
actions (later called the Zimmer program), Gromov \cite{Gr}
introduced the notion of rigid geometric structures, and proved
several deep results which have had profound influence on the
geometry and dynamics of group actions (see
\cite{Ad,BFL,BL,Fi,La,Zi93,ZM} and the references therein). For
example, for a connected $C^\varepsilon$ manifold $M$ with a rigid
$C^\varepsilon$ $\A$-structure, where $\varepsilon=\infty$ or
$\omega$, Gromov proved the following fundamental theorems.

\begin{itemize}
\item \emph{Gromov's open dense orbit theorem}:
If $\Iso^\germ(M)$ has a dense orbit in $M$, then it has an open
dense orbit.
\item \emph{Gromov's centralizer theorem}: Suppose that $\varepsilon=\omega$, $M$ is compact,
and a connected noncompact simple Lie group $G$ acts faithfully,
analytically, and isometrically on $M$ and preserves a finite smooth
measure on $M$. Let $\cZ$ be the centralizer of $\widetilde{\cG}$ in
$\Kill(\widetilde{M})$. Then for a.e. $\tilde{x}\in\widetilde{M}$,
we have $T_{\tilde{x}}\widetilde{G}\tilde{x}\subset
\ev_{\tilde{x}}(\cZ)$.
\item \emph{Gromov's representation theorem}: Under the conditions of Gromov's centralizer theorem,
if $\rho:\pi_1(M)\to\GL(\cZ)$ is the representation induced by the
deck transformations, then $\overline{\overline{\rho(\pi_1(M))}}$
has a Lie subgroup locally isomorphic to $G$.
\end{itemize}
Note that rigid geometric structures generalize Cartan's structures
of finite type, and rigid $\A$-structures include pseudo-Riemannian
structures, linear connections, pseudo-Riemannian conformal
structures in dimension $\ge3$, etc. The reader may refer to
\cite{DG,FK,Fi,La,Zi93,ZM} for surveys of Gromov's theory. For
detailed discussions and further developments of the above theorems,
see, e.g., \cite{BFM,Be,BF,CQ03,CQ04,Du,Fe,FZ,Me,NZ,Ze00,Ze02}.

\subsection{The general theorem}

The first goal of this paper is to prove a general theorem which, in a
certain sense, unifies and generalizes the above three theorems of
Gromov. To state the result, we need to recall and introduce more
notation and terminology. Let $M$ be a connected $C^\varepsilon$
manifold with a rigid $C^\varepsilon$ $\A$-structure $\sigma$, where
$\varepsilon=\infty$ or $\omega$. It is proved in \cite{Gr} (see
also \cite{Be,Fe}) that there exists an open dense subset
$M_\reg\subset M$ invariant under $\Iso^\germ(M)$ such that for some
sufficiently large integer $k$, every $x\in M_\reg$ has a
neighborhood $U_x$ such that any infinitesimal isometry of order $k$
sending $x$ to $y\in U_x$ extends to a local $C^\varepsilon$
isometry. Moreover, if $\varepsilon=\omega$ and $M$ is compact, we
may take $M_\reg=M$. Throughout this paper, we fix $M_\reg$ for a
given $(M,\sigma)$, and set $M_\reg=M$ if $\varepsilon=\omega$ and
$M$ is compact.

Let a connected Lie group $G$ act on $M$ by $C^\varepsilon$
isometries. For a finite Borel measure $\mu$ on $M$, we denote
$G_\mu=\{g\in G\mid g_*(\mu)=\mu\}$, which is a closed subgroup of
$G$ (see, e.g., \cite[Prop. 1.1]{Ra}). We say a subset $M'\subset M$
is \emph{$G$-saturated} if $x\in M'$, $y\in M$, and
$\overline{Gx}=\overline{Gy}$ imply that $y\in M'$. The interest of
this notion lies in the fact that if a $G$-minimal set $M_0\subset
M$ intersects a $G$-saturated set $M'$, then $M_0\subset M'$.

Let $\cV$ be a subspace of $\Kill(M)$. We define various types of
centralizers of $\cV$ as follows.
\begin{itemize}
\item We say that
$\varphi\in\Iso^\germ(M)$ centralizes $\cV$ if $d\varphi(\langle
v\rangle_x)=\langle v\rangle_y$ for all $v\in\cV$, where $x$ and $y$
are the source and target of $\varphi$, respectively. We refer to
the subgroupoid of $\Iso^\germ(M)$ consisting of all germs
centralizing $\cV$ as the \emph{centralizer of $\cV$ in
$\Iso^\germ(M)$}, and denote it by $\Iso^\germ(M)^\cV$.
\item For $x\in M$, we refer to the
centralizer in $\Kill^\germ_x(M)$ of the image of $\cV$ under the
natural injection $\Kill(M)\to\Kill^{\germ}_x(M)$ (see Theorem
\ref{T:Gromov1}(3) below) as the \emph{centralizer of $\cV$ in
$\Kill^\germ_x(M)$}, and denote it by $\Kill^{\germ}_x(M)^\cV$.
\item We refer to
the centralizer in $\Kill(\widetilde{M})$ of the image of $\cV$
under the natural injection $\Kill(M)\to\Kill(\widetilde{M})$ as the
\emph{centralizer of $\cV$ in $\Kill(\widetilde{M})$}, and denote it
by $\Kill(\widetilde{M})^\cV$.
\end{itemize}
Note that the space $\cZ$ in Gromov's centralizer and representation
theorems coincides with $\Kill(\widetilde{M})^\cG$.

Now we assume that $\cV$ is $G$-invariant under the natural
representation of $G$ in $\Kill(M)$. Then we have a representation
$\rho_\cV:G\to\GL(\cV)$, $g\mt dg|_\cV$. We are interested in the
$\rho_\cV$-discompactness of $G$. For the $\cV\supset\cG$ case, we
have the following criterion.

\begin{lemma}\label{L:discompact}
Let $M$ and $G$ be as above, and let $\cV$ be a $G$-invariant
subspace of $\Kill(M)$ containing $\cG$.
\begin{itemize}
\item[(1)] If the radical $R(G)$ of $G$ is
split solvable, the semisimple group $G/R(G)$ has no compact
factors, and there exists a sequence
$\cG=\cV_0\subset\cV_1\subset\cdots\subset\cV_k=\cV$ of
$R(G)$-invariant subspaces of $\cV$ with $\dim\cV_i/\cV_{i-1}=1$,
then $G$ is $\rho_\cV$-discompact.
\item[(2)] If the $G$-action on $M$ is faithful, then the converse of (1)
also holds.
\end{itemize}
\end{lemma}

Consider the natural representation of $\Gamma=\pi_1(M)$ in
$\Kill(\widetilde{M})$ induced by the deck transformations. Let
$\cW$ be a $\Gamma$-invariant subspace of $\Kill(\widetilde{M})$.
For a subgroupoid $\sA$ of $\Iso^\germ(M)$ and a subalgebra $\La_x$
of $\Kill^{\germ}_x(M)$, we will define the Zariski hulls of
$\Gamma$ in $\sA$ and $\La_x$ relative to $\cW$, denoted by
$\Hull^\cW_{\Gamma}(\sA)$ and $\Hull^\cW_{\Gamma}(\La_x)$
respectively, which encode some information of $\Gamma$. Here we
only note that $\Hull^\cW_{\Gamma}(\sA)$ is a subgroupoid of $\sA$,
$\Hull^\cW_{\Gamma}(\La_x)$ is a subalgebra of $\La_x$, and if
$\cW=0$, then $\Hull^\cW_{\Gamma}(\sA)=\sA$ and
$\Hull^\cW_{\Gamma}(\La_x)=\La_x$.

Our general theorem is as follows.

\begin{theorem}\label{T:main}
Let $M$ and $G$ be as above, $\cV$ be a $G$-invariant subspace of
$\Kill(M)$, $\cW$ be a $\Gamma$-invariant subspace of
$\Kill(\widetilde{M})^\cG$, and $\mu$ be a finite Borel measure on
$M$. Suppose that $G$ is $\rho_\cV$-discompact and
$\ker(\rho_\cV)\cdot G_\mu=G$. Then there exists a $G$-saturated
constructible set $M'\subset M_\reg$ with $\mu(M_\reg\backslash
M')=0$ such that for every $x\in M'$, we have
\begin{itemize}
\item[(1)] the
$\Hull_\Gamma^\cW(\Iso^\germ(M)^\cV)$-orbit of $x$ contains a
$G$-invariant open dense subset of $\overline{Gx}$, and
\item[(2)] $T_xGx\subset\ev_x(\Hull_\Gamma^\cW(\Kill^\germ_x(M)^\cV))$.
\end{itemize}
\end{theorem}

\subsection{The centralizer and representation theorems}

The $\cW=0$ case of Theorem \ref{T:main}(1) can be viewed as a
``global" version of Gromov's centralizer theorem, which asserts
that the $\Iso^\germ(M)^\cV$-orbit of $x$ contains not only $Gx$ but
also an open dense subset of $\overline{Gx}$. Thus it generalizes Gromov's open dense orbit theorem (see Theorem \ref{T:opendense} below).
Theorem \ref{T:main}(2) is the
``infinitesimal" version of Theorem \ref{T:main}(1) and will be derived from it. By taking $\cW=0$ in Theorem \ref{T:main}(2), one easily obtain the
following Theorem \ref{T:cen-rep}(1), which generalizes Gromov's
centralizer theorem. But Theorem \ref{T:main}(2) indeed implies that
for any $\Gamma$-invariant subspace $\cW$ of
$\Kill(\widetilde{M})^\cG$, the first assertion of Theorem
\ref{T:cen-rep}(1) remains true if we replace $\Kill^\germ_x(M)^\cV$
by its subalgebra $\Hull_\Gamma^\cW(\Kill^\germ_x(M)^\cV)$. By
taking $\cW=\Kill(\widetilde{M})^\cV$, we obtain Theorem
\ref{T:cen-rep}(2) in a very direct way, which generalizes Gromov's
representation theorem.

\begin{theorem}\label{T:cen-rep}
Let $M$ be a connected $C^\varepsilon$ manifold with a rigid
$C^\varepsilon$ $\A$-structure and a finite Borel measure $\mu$, $G$
be a connected Lie group which acts on $M$ by $C^\varepsilon$
isometries and preserves $\mu$, and $\cV\supset\cG$ be a
$G$-invariant subspace of $\Kill(M)$ such that $G$ is
$\rho_\cV$-discompact. Then there exists a $G$-saturated
constructible set $M'\subset M_\reg$ with $\mu(M_\reg\backslash
M')=0$ such that for every $x\in M'$, the following assertions hold.
\begin{itemize}
\item[(1)] $T_xGx\subset\ev_x(\Kill^\germ_x(M)^\cV)$ and
$\Lg(x)\lhd\Lg$. If moreover $\varepsilon=\omega$, then
$T_{\tilde{x}}\widetilde{G}\tilde{x}\subset
\ev_{\tilde{x}}(\Kill(\widetilde{M})^\cV)$ for every
$\tilde{x}\in\pi^{-1}(x)$.
\item[(2)] If $\varepsilon=\omega$ and
$\rho:\pi_1(M)\to\GL(\Kill(\widetilde{M})^\cV)$ is the
representation induced by the deck transformations, then
$\ad(\Lg/\Lg(x))\prec\LL\left(\overline{\overline{\rho(\pi_1(M))}}\right)$.
\end{itemize}
\end{theorem}

\begin{remark}\label{R:cen-rep}
\begin{itemize}
\item[(1)] The case of Theorem \ref{T:cen-rep} where $G$ is semisimple without compact factors
strengthens Gromov's centralizer and representation theorems.
Indeed, Lemma \ref{L:discompact} implies that $G$ is
$\rho_\cV$-discompact for $\cV=\Kill(M)$. Since
$\Kill(\widetilde{M})^{\Kill(M)}\subset\Kill(\widetilde{M})^\cG$,
the last assertion of Theorem \ref{T:cen-rep}(1) strengthens
Gromov's centralizer theorem. On the other hand, as is well-known,
if $\mu$ is smooth and the $G$-action is faithful, then the
$G$-action is a.e. locally free (see Lemma \ref{L:locallyfree}). In
particular, there exists $x\in M'$ such that $\Lg(x)=0$. Hence
Theorem \ref{T:cen-rep}(2) implies that
$\overline{\overline{\rho(\pi_1(M))}}$ has a Lie subgroup locally
isomorphic to $G$, where the representation space of $\rho$ is
$\Kill(\widetilde{M})^{\Kill(M)}$.
\item[(2)] The conditions of Theorem \ref{T:cen-rep} can be also
satisfied for $\cV=\Kill(M)$ and non-semisimple $G$. Indeed, if the
geometric structure is unimodular (e.g., a pseudo-Riemannian
structure or a linear connection plus a volume density), then there
exists a connected closed normal subgroup $G$ of $\Iso(M)$ such that
$\Iso(M)/G$ is locally isomorphic to a compact Lie group, and such
that the conditions of Theorem \ref{T:cen-rep} are satisfied for
$\cV=\Kill(M)$ (see Remark \ref{R:cocompact}).
\item[(3)] If $R(G)$ is split solvable and $G/R(G)$ has no
compact factors, then by Lemma \ref{L:discompact}, $G$ is
$\rho_\cG$-discompact. In this general case, the $G$-action may be
not a.e. locally free. But Theorem \ref{T:cen-rep}(2) relates the
size of $\overline{\overline{\rho(\pi_1(M))}}$ to the sizes of the
stabilizers. This is sufficient for our applications (Theorems
\ref{T:large} and \ref{T:fixedpoint2} below).
\item[(4)] If $G$ is split solvable, then it is amenable and $\rho_\cG$-discompact. Thus we
can deduce from Theorem \ref{T:cen-rep} the centralizer and
representation theorems for split solvable $G$ which dispenses with
the finite invariant measure (see Corollary \ref{C:cen-rep2}).
\item[(5)] We do not need $\mu$ to be smooth (or ``Zariski" in the sense
of \cite{CQ03}). But if $\mu$ is smooth, then $M'$ contains a
$G$-invariant open dense subset of $M$, due to the fact that a
constructible subset of a topological space contains an open dense
subset of its closure. (This is well-known for Noetherian spaces. We
will prove it for general topological spaces in Lemma
\ref{L:constructible}.)
\item[(6)] The proofs of Gromov's representation theorem and Theorem
\ref{T:cen-rep}(2) are based on the fact that if
$\varepsilon=\omega$, then every local Killing field on
$\widetilde{M}$ can be extended to a global Killing field. By using
Theorem \ref{T:main}(2), we can prove that if $\varepsilon=\infty$ and Gromov's representation
theorem does not hold, then the local Killing fields on
$\widetilde{M}$ must be highly non-extendable (see Theorem \ref{T:rep-smooth} below).
\end{itemize}
\end{remark}

\subsection{Applications of Theorem \ref{T:cen-rep}}

Our first application of Theorem \ref{T:cen-rep} is a structural
property of $\Iso(M)$ for a simply connected compact analytic $M$
with a rigid unimodular $\A$-structure. A theorem of D'Ambra
\cite{DA} asserts that if $M$ is a simply connected compact analytic
Lorentz manifold then $\Iso(M)$ is compact. Although D'Ambra's
theorem no longer holds for general pseudo-Riemannian structures
(see \cite[Sect. 5]{DA} for a counterexample), Gromov \cite[Thm.
3.7.A]{Gr} proved the following result, which we call it Gromov's
compactness theorem.
\begin{itemize}
\item \emph{Gromov's compactness theorem}:
Let $M$ be a connected and simply connected compact analytic
manifold with a rigid unimodular analytic $\A$-structure. Then
$\Iso(M)$ is a compact extension of a connected abelian group.
\end{itemize}
We will prove the following result, which imposes another
restriction on the possible choices of $\Iso(M)$ for $M$ as in
Gromov's compactness theorem (e.g., the special Euclidean group
$\SO(n)\ltimes\RR^n$ can not serve as $\Iso(M)_0$).

\begin{theorem}\label{T:Gromov}
Let $M$ be a connected and simply connected compact analytic
manifold with a rigid unimodular analytic $\A$-structure. If
$\Iso(M)_0$ is nontrivial, then it either is compact semisimple or has a
non-discrete center.
\end{theorem}

\begin{remark}\label{R:Gromov}
If the geometric structure is a pseudo-Riemannian structure or a
linear connection plus a volume density, the compactness condition
for $M$ in Theorem \ref{T:Gromov} can be replaced by the condition
that $M$ is geodesically complete and has finite volume (see Remark
\ref{R:complete}).
\end{remark}

The second application concerns the relationship between the
structure of $G$ and the fundamental group of $M$. Let $M$ be a
connected analytic manifold with a rigid unimodular analytic
$\A$-structure of finite volume, and let a connected Lie group $G$
act faithfully, analytically, and isometrically on $M$. It is
commonly believed that if $G$ is ``large", then $\pi_1(M)$ must also
be ``large". A typical example is Gromov's representation theorem,
which implies that if $G$ is simple, then $\pi_1(M)$ is
non-amenable. For more results in this direction, see
\cite{Fi,Zi93,ZM} and the references therein. By using Theorem
\ref{T:cen-rep}, we will prove the following theorem, which provides
three results in this spirit for split solvable $G$.

\begin{theorem}\label{T:large}
Let $M$ be a connected analytic manifold with a rigid unimodular
analytic $\A$-structure of finite volume, and let a connected split
solvable Lie group $G$ act faithfully, analytically, and
isometrically on $M$.
\begin{itemize}
\item[(1)] If $G$ is non-abelian, then $\pi_1(M)$ is infinite.
\item[(2)] If $G$ is not at most $2$-step nilpotent, then $\pi_1(M)$ is not virtually abelian.
\item[(3)] If $G$ is non-nilpotent, then $\pi_1(M)$ is not virtually nilpotent.
\end{itemize}
\end{theorem}

\begin{remark}\label{R:large}
\begin{itemize}
\item[(1)] The case of Theorem \ref{T:large}(1) where $M$ is compact
can be also deduced from Gromov's compactness theorem.
\item[(2)] Recall that a group is \emph{virtually abelian} (resp.
\emph{virtually nilpotent}) if it has an abelian (resp. nilpotent)
subgroup of finite index. By Gromov's polynomial growth theorem
\cite{Gr81}, a finitely generated group (e.g., $\pi_1(M)$ for a
compact $M$) is virtually nilpotent if and only if it has polynomial
growth.
\end{itemize}
\end{remark}

We also use Theorem \ref{T:cen-rep} to prove two fixed point
theorems for isometric actions of split solvable Lie groups. As is
well-known, solvable group actions tend to have fixed points. For
example, continuous affine actions of a solvable lcsc group (which
is amenable) on a compact convex subset of a locally convex
topological vector space have fixed points (see \cite{Zi84}). As
another example, the Borel fixed point theorem \cite{Bo91} asserts
that $\KK$-regular actions of a connected $\KK$-split solvable
$\KK$-group on a complete $\KK$-variety $\bV$ with
$\bV_\KK\ne\emptyset$ have fixed points in $\bV_\KK$. For isometric
actions of a split solvable Lie group, we will prove the following
result.

\begin{theorem}\label{T:fixedpoint1}
Let $M$ be a connected and simply connected compact analytic
manifold with a rigid non-unimodular analytic $\A$-structure, and
let $G$ be a connected split solvable Lie subgroup of $\Iso(M)$.
Suppose that
\begin{itemize}
\item[(1)] $G$ is discompact with respect to the restriction to $G$ of the adjoint
representation of $\Iso(M)$, and
\item[(2)] $\Iso(M)_0$ has a discrete center.
\end{itemize}
Then any $G$-minimal set in $M$ consists of a single point. In
particular, $G$ has a fixed point in $M$.
\end{theorem}

Under much weaker conditions, especially without any assumption on
$\Iso(M)_0$, we can prove that the commutator group $(G,G)$ always
has a fixed point.

\begin{theorem}\label{T:fixedpoint2}
Let $M$ be a connected compact analytic manifold with a rigid
non-unimodular analytic $\A$-structure, and let $G$ be a connected
split solvable Lie group which acts analytically and isometrically
on $M$. Suppose that either
\begin{itemize}
\item[(1)] $\pi_1(M)$ is finite, or
\item[(2)] $\pi_1(M)$ is virtually nilpotent and
$[\Lg,[\Lg,\Lg]]=[\Lg,\Lg]$.
\end{itemize}
Then $(G,G)$ acts trivially on any $G$-minimal set in $M$. In
particular, $(G,G)$ has a fixed point in $M$.
\end{theorem}

\begin{remark}
\begin{itemize}
\item[(1)] Condition (1) in Theorem \ref{T:fixedpoint1} is satisfied if $H$ is
a connected noncompact semisimple Lie subgroup of $\Iso(M)$ with an
Iwasawa decomposition $H=KAN$ and $G=AN$ (see Remark
\ref{R:fixedpoint1}).
\item[(2)] Many important split solvable Lie groups $G$ satisfy the condition
$[\Lg,[\Lg,\Lg]]=[\Lg,\Lg]$. The simplest example is the ``$ax+b$"
group. Another typical class of examples is $G=AN$, where $KAN$ is
an Iwasawa decomposition of a noncompact semisimple Lie group (in
this case we have $(G,G)=N$). Note that if $G$ satisfies the
condition, then so does any quotient group of $G$.
\item[(3)] The non-unimodularity of the geometric structures in Theorems
\ref{T:fixedpoint1} and \ref{T:fixedpoint2} will not be used in the
proofs. But Theorems \ref{T:Gromov} and \ref{T:large} imply that
Theorems \ref{T:fixedpoint1} and \ref{T:fixedpoint2} are meaningful
only for non-unimodular structures.
\end{itemize}
\end{remark}

\subsection{The main idea in the proof and algebraic quotients}

As is well-known, the notion of algebraic hulls \cite{Zi84} is
widely used in the study of Gromov's centralizer theorem and related
topics. However, it seems to the author that the algebraic hull
method is not sufficient to prove Theorem \ref{T:main} in full
generality. Our proof of Theorem \ref{T:main} uses some idea in
Gromov's proof of his centralizer theorem outlined in \cite[Sec.
5.2]{Gr}. But there are essential improvements. We construct a
continuous geometric structure by augmenting $\sigma$ with the
geometric structures associated with the jets of Killing fields in
$\cV$ and $\cW$. The geometric structure associated with $\cW$ has
its range in an algebraic quotient endowed with the quotient
topology, and can be only expected to be continuous. But it has the
advantage that it relates directly to $\pi_1(M)$. Then we consider
the Gauss map of the augmented geometric structure, which also has
range in an algebraic quotient and is continuous. Although the
quotient topology on the algebraic quotient is in general not
Hausdorff, the continuity of the Gauss map enables us to investigate
interesting topological properties.

A crucial tool in the proof of Theorem \ref{T:main} is a Borel
density theorem for algebraic quotients. Let $\KK$ be a locally
compact non-discrete field, and let $\bG$ be a $\KK$-group.
Following \cite{Sh}, we say that $\bG$ is \emph{$\KK$-discompact} if
any $\KK$-homomorphism $\bG\to\bH$ of $\KK$-groups with $\bH_\KK$
compact is trivial. One version of the Borel density theorem is as
follows (see \cite{Sh}).
\begin{itemize}
\item \emph{Borel density
theorem}: If a $\KK$-discompact $\KK$-group $\bG$ acts
$\KK$-regularly on a $\KK$-variety $\bV$ with $\bV_\KK\ne\emptyset$
and $\bG_\KK$ preserves a finite Borel measure $\mu$ on $\bV_\KK$,
then $\mu$-a.e. point in $\bV_\KK$ is fixed by $\bG$.
\end{itemize}
Now we assume that $\mathrm{char}\,\KK=0$. Let $\bG$ and $\bH$ be
$\KK$-groups, and let $\bV$ be a $\KK$-variety with
$\bV_\KK\ne\emptyset$. Suppose that $\bG\times\bH$ acts
$\KK$-regularly on $\bV$. We endow the algebraic quotient
$\bV_\KK/\bH_\KK$ with the quotient Borel structure. Then $\bG_\KK$
acts measurably on $\bV_\KK/\bH_\KK$. We will prove and use the
following result.

\begin{theorem}\label{T:Borel}
Let $\bG$, $\bH$, and $\bV$ be as above such that $\bG$ is
$\KK$-discompact, and let $\mu$ be a finite Borel measure on
$\bV_\KK/\bH_\KK$ such that either
\begin{itemize}
\item[(1)] $\mu$ is $\bG_\KK$-invariant, or
\item[(2)] $\KK=\RR$ and $\mu$ is invariant under some Zariski dense Lie subgroup of $\bG_\RR$.
\end{itemize}
Then for $\mu$-a.e. $x\in\bV_\KK/\bH_\KK$, $\bG_\KK(x)$ is an open
subgroup of finite index in $\bG_\KK$.
\end{theorem}

It is necessary to point out the following difference between
Theorem \ref{T:Borel} and the usual Borel density theorem, which
already occurs for $\KK=\RR$: Although the $\KK=\RR$ (and $\mu$ is
$\bG_\RR$-invariant) case of Theorem \ref{T:Borel} implies that the
set of $(\bG_\RR)_0$-fixed points in $\bV_\RR/\bH_\RR$ is
$\mu$-conull, $\bG_\RR$ (which is Zariski connected by the
discompactness) may have no fixed points in $\bV_\RR/\bH_\RR$ (see
Remark \ref{R:counterexample}).

\subsection{Organization of the paper}

In Section \ref{S:auxiliary} we prove some auxiliary results
concerning constructible sets, $G$-saturated sets, and subquotients
of Lie algebras. In Section \ref{S:review} we briefly review
Gromov's theory of rigid geometric structures. In Section
\ref{S:Zariski} we define Zariski hulls of fundamental groups and
prove some of their basic properties. In Section \ref{S:discompact}
we study discompact groups and prove Lemma \ref{L:discompact}. Then
in Section \ref{S:Borel} we prove Theorem \ref{T:Borel}. The proof
is based on the usual Borel density theorem. Theorem \ref{T:main} is
proved in Section \ref{S:proof}. We first prove Theorem
\ref{T:main}(1) by using Theorem \ref{T:Borel}, and then prove
Theorem \ref{T:main}(2) by passing from global to infinitesimal
information. In Section \ref{S:proof1} we provide the details of how Theorem
\ref{T:main} unifies Gromov's theorems, in particular, the proof of Theorem \ref{T:cen-rep}.
Theorems \ref{T:Gromov}--\ref{T:fixedpoint2} are proved in Section
\ref{S:proof2}.

\section{Auxiliary results}\label{S:auxiliary}

In this section we prove some basic properties of constructible
sets, $G$-saturated sets, and subquotients of Lie algebras. The
reader may skip this section during the first reading and return to
it later as reference.

\subsection{Constructible sets}

Let $X$ be a topological space. Recall that a subset of $X$ is
\emph{constructible} if it is a finite union of locally closed sets.
It is well-known that if $X$ is Noetherian, then every constructible
subset of $X$ contains an open dense subset of its closure
(\cite[AG.1.3]{Bo91}). We prove that this remains true for arbitrary
$X$.

\begin{lemma}\label{L:constructible}
Let $X$ be a topological space. If $Y\subset X$ is constructible,
then it contains an open dense subset of $\overline{Y}$.
\end{lemma}

\begin{proof}
Suppose $Y=\bigcup_{i=1}^kY_i$, where each $Y_i$ is locally closed.
Denote $Z_i=\overline{Y_i}\,\backslash Y_i$, $Z=\bigcup_{i=1}^kZ_i$.
We prove that $W=\overline{Y}\,\backslash Z$ is contained in $Y$ and
is open dense in $\overline{Y}$. We first prove that $W\subset Y$.
Since
$$Y\cup Z=\left(\bigcup_{i=1}^kY_i\right)\cup\left(\bigcup_{i=1}^kZ_i\right)
=\bigcup_{i=1}^k(Y_i\cup
Z_i)=\bigcup_{i=1}^k\overline{Y_i}=\overline{\bigcup_{i=1}^kY_i}=\overline{Y},$$
we have $W=\overline{Y}\,\backslash Z=(Y\cup Z)\backslash
Z=Y\backslash Z\subset Y$. Now we show that $W$ is open in
$\overline{Y}$. Since $Y_i$ is locally closed, it is open in
$\overline{Y_i}$. So $Z_i$ is closed in $\overline{Y_i}$, and hence
is closed in $X$. Thus $Z$ is closed, and hence $W$ is open in
$\overline{Y}$. Finally we prove that $W$ is dense in
$\overline{Y}$. If not, then $Z$ contains a nonempty open subset of
$\overline{Y}$. Let $i_0$ be the minimal index such that
$\bigcup_{i=1}^{i_0}Z_i$ contains a nonempty open subset, say $U$,
of $\overline{Y}$. We claim that $U\not\subset Z_i$ for any $1\le
i\le i_0$. Indeed, if $U\subset Z_i$, then $U$ is a nonempty open
subset of $\overline{Y_i}$ with $U\cap Y_i=\emptyset$, which is
impossible. This in particular implies that  $i_0>1$ and
$U\not\subset Z_{i_0}$. So $\emptyset\ne U\backslash
Z_{i_0}\subset\bigcup_{i=1}^{i_0-1}Z_i$. Since $Z_{i_0}$ is closed,
$U\backslash Z_{i_0}$ is open in $\overline{Y}$. This conflicts with
the minimality of $i_0$. Thus $W$ is dense in $\overline{Y}$.
\end{proof}

\begin{corollary}\label{C:constructible}
Let a Lie group $G$ act smoothly on a smooth manifold $M$, $M_\reg$
be an open dense subset of $M$, and $M'\subset M_\reg$ be a
$G$-invariant constructible set. If there exists a smooth measure
$\mu$ on $M$ such that $\mu(M_\reg\backslash M')=0$, then $M'$
contains a $G$-invariant open dense subset of $M$.
\end{corollary}

\begin{proof}
Since $\mu$ is smooth and $\mu(M_\reg\backslash M')=0$, $M'$ is
dense in $M_\reg$, and hence is dense in $M$. Since $M'$ is
constructible, by Lemma \ref{L:constructible}, $M'$ contains an open
dense subset $U$ of $M$. But $M'$ is $G$-invariant. So the
$G$-invariant open dense subset $\bigcup_{g\in G}gU$ of $M$ is
contained in $M'$.
\end{proof}

\begin{lemma}\label{L:Fconstructible}
Let $G$ and $H$ be Lie groups with $G$ connected, and let $G\times
H$ act smoothly on a smooth manifold $M$. Consider the induced
$G$-action on $M/H$. Then $(M/H)^G$ is constructible with respect to
the quotient topology on $M/H$.
\end{lemma}

\begin{proof}
It suffices to prove that the preimage $P$ of $(M/H)^G$ under the
quotient map $M\to M/H$ is of the form $\bigcup_{i=0}^kU_i\cap V_i$,
where $U_i$ is $H$-invariant and open, $V_i$ is $H$-invariant and
closed. If we denote $K=G\times H$, it is easy to see that
$$P=\{x\in M\mid Kx=Hx\}.$$ For $0\le i\le\dim H$, we denote
$$U_i=\{x\in M\mid\dim Hx\ge i\}, \quad
V_i=\{x\in M\mid\dim Kx\le i\}.$$ Then $U_i$ is $H$-invariant and
open, $V_i$ is $H$-invariant and closed. We prove the lemma by
showing that $P=\bigcup_{i=0}^{\dim H}U_i\cap V_i$. It is obvious
that $P\subset\bigcup_{i=0}^{\dim H}U_i\cap V_i$. To prove the
converse, suppose that $x\in U_i\cap V_i$ for some $i$. Let
$p_1:K\to G$ be the projection. Then
\begin{align*}
\dim p_1(K(x))&=\dim K(x)-\dim H(x)\\
&=(\dim K-\dim Kx)-(\dim H-\dim Hx)\\
&\ge(\dim K-i)-(\dim H-i)=\dim G.
\end{align*}
Since $G$ is connected, we must have $p_1(K(x))=G$. This means that
$H\cdot K(x)=K$. Thus $Kx=Hx$, and hence $x\in P$. This completes
the proof.
\end{proof}

\begin{remark}
The set $(M/H)^G$ need not to be locally closed. For example, if
$G=H=\GL(n,\CC)$, $M$ is the space of $n\times n$ complex matrices,
and $G\times H$ acts on $M$ by $(g,h)A=gAh^{-1}$, then the preimage
of $(M/H)^G$ under $M\to M/H$ is equal to $\GL(n,\CC)\cup\{0\}$,
which is not locally closed. Hence $(M/H)^G$ is not locally closed.
\end{remark}

\subsection{$G$-saturated sets}

Let a topological group $G$ act continuously on a topological space
$X$. A subset $Y\subset X$ is \emph{$G$-saturated} if $x\in X$,
$y\in Y$ and $\overline{Gx}=\overline{Gy}$ imply that $x\in Y$.
Obviously, $G$-saturated sets are $G$-invariant.

\begin{lemma}\label{L:boolean}
Let $G$ and $X$ be as above. Then the family of $G$-saturated
subsets of $X$ is closed under intersection, union, and complement.
\end{lemma}

\begin{proof}
We define an equivalence relation on $M$ by setting $x\sim y$ if and
only if $\overline{Gx}=\overline{Gy}$. Then a subset of $X$ is
$G$-saturated if and only if it is the union of some equivalence
classes. Hence the lemma follows.
\end{proof}

\begin{lemma}\label{L:saturated}
Let $G$ and $X$ be as above. Then any $G$-invariant locally closed
subset of $X$ is $G$-saturated.
\end{lemma}

\begin{proof}
Obviously, $G$-invariant closed sets are $G$-saturated. Then Lemma
\ref{L:boolean} implies that $G$-invariant open sets are also
$G$-saturated. By Lemma \ref{L:boolean} again, it suffices to prove
that any $G$-invariant locally closed set $Y\subset X$ is the
intersection of a $G$-invariant open set and a $G$-invariant closed
set. Suppose that $Y=U\cap V$, where $U$ is open, $V$ is closed. Let
$U'=\bigcup_{g\in G}gU$, $V'=\bigcap_{g\in G}gV$. Then $U'$ is
$G$-invariant and open, $V'$ is $G$-invariant and closed. We claim
that $Y=U'\cap V'$. Indeed, since $Y$ is $G$-invariant, we have
$$Y=\bigcap_{g\in G}gY=\bigcap_{g\in G}(gU\cap
gV)=\left(\bigcap_{g\in G}gU\right)\cap\left(\bigcap_{g\in
G}gV\right)\subset U'\cap V'$$ and $$Y=\bigcup_{g\in
G}gY=\bigcup_{g\in G}(gU\cap gV)\supset\bigcup_{g\in G}(gU\cap
V')=\left(\bigcup_{g\in G}gU\right)\cap V'=U'\cap V'.$$ This
completes the proof.
\end{proof}

\begin{lemma}\label{L:inverseimage}
Let a topological group $G$ act continuously on two topological
spaces $X_1$ and $X_2$, and let $\theta:X_1\to X_2$ be a
$G$-equivariant continuous map. If $Y\subset X_2$ is $G$-saturated,
then so is $\theta^{-1}(Y)$.
\end{lemma}

\begin{proof}
Suppose $x\in X_1$, $y\in\theta^{-1}(Y)$ and
$\overline{Gx}=\overline{Gy}$. Since $Y$ is $G$-saturated,
$\theta(y)\in Y$, and
$$\overline{G\theta(x)}=\overline{\theta(Gx)}=\overline{\theta(\overline{Gx})}
=\overline{\theta(\overline{Gy})}=\overline{\theta(Gy)}=\overline{G\theta(y)},$$
we have $\theta(x)\in Y$. Hence $x\in\theta^{-1}(Y)$. This completes
the proof.
\end{proof}

\begin{lemma}\label{L:minimal-saturated}
Let a Lie group $G$ act smoothly on a compact smooth manifold $M$,
$M_0\subset M$ be a $G$-minimal set, and $\mu$ be a $G$-invariant
Borel measure on $M$ supported on $M_0$. If $M'\subset M$ is a
$G$-saturated Borel set with $\mu(M')>0$, then $M_0\subset M'$.
\end{lemma}

\begin{proof}
Since $\mu$ is supported on $M_0$ and $\mu(M')>0$, we have $M_0\cap
M'\ne\emptyset$. Choose $x_0\in M_0\cap M'$, and let $x\in M_0$.
Since $M_0$ is $G$-minimal, we have $\overline{Gx}=\overline{Gx_0}$.
So $x\in M'$. Hence $M_0\subset M'$.
\end{proof}

\subsection{Subquotients of Lie algebras}

For simplicity, we assume that all Lie algebras in this subsection
are real and finite-dimensional.

\begin{lemma}\label{L:subquotient1}
Let $\La,\Lb,\Lc$ be Lie algebras. If $\La\prec\Lb$ and
$\Lb\prec\Lc$, then $\La\prec\Lc$.
\end{lemma}

\begin{proof}
By definition, there exist $\Lb_1<\Lb$, $\Lc_1<\Lc$, and surjective
homomorphisms $\varphi:\Lb_1\to\La$, $\psi:\Lc_1\to\Lb$. Let
$\Lc_2=\psi^{-1}(\Lb_1)$. Then
$\varphi\circ\psi|_{\Lc_2}:\Lc_2\to\La$ is a surjective
homomorphism. Thus $\La\prec\Lc$.
\end{proof}

\begin{lemma}\label{L:subquotient2}
Let $\La$, $\Lb$ be Lie algebras. If $\La\prec\Lb$, then
$\ad(\La)\prec\ad(\Lb)$.
\end{lemma}

\begin{proof}
Let $\varphi:\Lb_1\to\La$ be a surjective homomorphism, where
$\Lb_1<\Lb$. Since $\varphi(Z(\Lb_1))\subset Z(\La)$, we have
\begin{align*}
\ad(\La)&\cong\La/Z(\La)\prec\La/\varphi(Z(\Lb_1))\cong\Lb_1/\varphi^{-1}(\varphi(Z(\Lb_1)))\\
&\prec\Lb_1/Z(\Lb_1)\prec\Lb_1/(Z(\Lb)\cap\Lb_1)<\Lb/Z(\Lb)\cong\ad(\Lb).
\end{align*}
This completes the proof.
\end{proof}

\begin{lemma}\label{L:subquotient3}
Let $\Lf$ be a Lie algebra, and let $\Lg, \Lh<\Lf$ be such that
$[\Lg,\Lh]=0$ and $\Lf=\Lg+\Lh$. Let $V$ be a real vector space,
$\ell:\Lf\to V$ be a linear map such that $\ker(\ell)<\Lf$ and
$\Lf=\Lh+\ker(\ell)$. Then there exist $\Lz<\Lh$ and a Lie algebra
structure on $\ell(\Lg)$ such that $\ell(\Lz)=\ell(\Lg)$, and such
that $\ell|_\Lg:\Lg\to \ell(\Lg)$ and $-\ell|_\Lz:\Lz\to \ell(\Lg)$
are Lie algebra homomorphisms. In particular, we have
$\ell(\Lg)\prec\Lh$.
\end{lemma}

\begin{proof}
Let $\Lk=\ker(\ell)$, $\Lz=\Lh\cap(\Lg+\Lk)$. Then $\ell(\Lz)\subset
\ell(\Lg+\Lk)=\ell(\Lg)$. On the other hand, it is easy to see that
$\Lg=\Lg\cap(\Lh+\Lk)\subset\Lz+\Lk$. So
$\ell(\Lg)\subset\ell(\Lz+\Lk)=\ell(\Lz)$. Hence
$\ell(\Lz)=\ell(\Lg)$.

We claim that $\Lg\cap\Lk\lhd\Lf$. Indeed, since $[\Lg,\Lh]=0$ and
$\Lf=\Lg+\Lh$, we have $\Lg\lhd\Lf$, and hence $\Lg\cap\Lk\lhd\Lk$.
Thus
$[\Lf,\Lg\cap\Lk]=[\Lh,\Lg\cap\Lk]+[\Lk,\Lg\cap\Lk]\subset\Lg\cap\Lk$.
This verifies the claim. Note that $\ker(\ell|_\Lg)=\Lg\cap\Lk$. So
there exists a Lie algebra structure on $\ell(\Lg)$ such that
$\ell|_\Lg$ is a homomorphism. It remains to prove that $-\ell|_\Lz$
is a homomorphism. Let $Z_i\in\Lz$, $i=1,2$. Then $Z_i=X_i+Y_i$ for
some $X_i\in\Lg$ and $Y_i\in\Lk$. Since
\begin{align*}
[Z_1,Z_2]&=[X_1,X_2]+[Y_1,Y_2]+[X_1,Y_2]+[Y_1,X_2]\\
&=[X_1,X_2]+[Y_1,Y_2]+[X_1,Z_2-X_2]+[Z_1-X_1,X_2]\\
&=-[X_1,X_2]+[Y_1,Y_2],
\end{align*}
we have
$$-\ell([Z_1,Z_2])=\ell([X_1,X_2])=[\ell(X_1),\ell(X_2)]=[-\ell(Z_1),-\ell(Z_2)].$$
Thus $-\ell|_\Lz$ is a homomorphism.
\end{proof}

\section{Review of geometric structures}\label{S:review}

In this section, we briefly review some definitions and facts in
Gromov's theory of rigid geometric structures \cite{Gr}. Let $n$ and
$r$ be positive integers. We denote by $\GL^r(n)$ the real algebraic
group of $r$-jets at $0\in\RR^n$ of diffeomorphisms of $\RR^n$ that
fix $0$. Let $M$ be a connected $n$-dimensional $C^\varepsilon$
manifold, where $\varepsilon=\infty$ or $\omega$. We denote by
$F^r(M)$ the $r$-th order frame bundle of $M$. This is a principal
$\GL^r(n)$-bundle over $M$, whose fiber $F^r(M)_x$ at $x\in M$
consists of the $r$-jets at $0\in\RR^n$ of diffeomorphisms from some
neighborhood of $0$ in $\RR^n$ into $M$ that send $0$ to $x$. Let
$V$ be a topological space, and let $\GL^r(n)$ act continuously on
$V$. A \emph{continuous geometric structure} of order $r$ and type
$V$ on $M$ is a continuous $\GL^r(n)$-equivariant map
$\sigma:F^r(M)\to V$. We refer to the induced continuous map
$\theta:M\to V/\GL^r(n)$ as the \emph{Gauss map} of $\sigma$, where
$V/\GL^r(n)$ is endowed with the quotient topology. We say that
$\sigma$ is $C^\varepsilon$ if $V$ is a $C^\varepsilon$
$\GL^r(n)$-manifold and $\sigma$ is a $C^\varepsilon$ map, and say
that $\sigma$ is a \emph{structure of algebraic type}
(\emph{$\A$-structure}, for short) if $V$ is a real algebraic
manifold and the $\GL^r(n)$-action on $V$ is regular. Whenever
$\sigma$ is $C^\varepsilon$, we say that it is \emph{unimodular} if
there is a preassigned $\GL^r(n)$-equivariant $C^\varepsilon$ map
$\delta:V\to(0,+\infty)$, where $\GL^r(n)$ acts on $(0,+\infty)$ by
$\alpha.t=|\det(\pi^r_1(\alpha))|t$, here $\alpha\in\GL^r(n)$,
$t\in(0,+\infty)$, and $\pi^r_1:\GL^r(n)\to\GL(n,\RR)$ is the
natural projection. The composition $\delta\circ\sigma$ induces a
$C^\varepsilon$ volume density $F^1(M)\to(0,+\infty)$, and hence a
smooth measure on $M$.

Now we assume that $\sigma:F^r(M)\to V$ is a $C^\varepsilon$
geometric structure. Let $s\ge0$ be an integer, and let $J_n^s(V)$
denote the space of $s$-jets at $0\in\RR^n$ of $C^\varepsilon$ maps
from some neighborhood of $0$ in $\RR^n$ into $V$. Then
$\GL^{r+s}(n)$ acts naturally on $J_n^s(V)$. The \emph{$s$-th
prolongation} $\sigma^s:F^{r+s}(M)\to J_n^s(V)$ of $\sigma$ is a
$C^\varepsilon$ geometric structure of order $r+s$ (see
\cite{CQ03,Fe}). If $\sigma$ is an $\A$-structure, then so is
$\sigma^s$.

If $f$ is a $C^\varepsilon$ map from a neighborhood of $x\in M$ to
some manifold or is the germ at $x$ of such a map, we denote by
$j_x^if$ the $i$-jet of $f$ at $x$, where $i\ge0$ is an integer. Let
$f$ be a $C^\varepsilon$ diffeomorphism from some open set $U\subset
M$ into $M$. We say that $f$ is a \emph{local $C^\varepsilon$
isometry} if $\sigma(j^r_xf\circ\beta)=\sigma(\beta)$ for all $x\in
U$ and $\beta\in F^r(M)_x$. If moreover $U=f(U)=M$, then $f$ is
called a \emph{$C^\varepsilon$ isometry}. Note that if $\sigma$ is
unimodular, then the induced smooth measure on $M$ is preserved by
the $C^\varepsilon$ isometry group $\Iso(M)$. For $i\ge r$ and $x\in
U$, if $\sigma^{i-r}(j^i_xf\circ\beta)=\sigma^{i-r}(\beta)$ for all
$\beta\in F^i(M)_x$, then $j^i_xf$ is called an \emph{infinitesimal
isometry} of order $i$. We denote by $\Iso^{\germ}(M)$ and
$\Iso^i(M)$ the groupoids of germs of local $C^\varepsilon$
isometries and infinitesimal isometries of order $i$ respectively,
denote by $\Iso^{\germ}_{x,y}(M)$ (resp. $\Iso^i_{x,y}(M)$) the
subset of $\Iso^{\germ}(M)$ (resp. $\Iso^i(M)$) consisting of
elements with source $x$ and target $y$, and denote
$\Iso^{\germ}_x(M)=\Iso^{\germ}_{x,x}(M)$,
$\Iso^i_x(M)=\Iso^i_{x,x}(M)$. Then we have natural maps
$\Iso^{\germ}_{x,y}(M)\to\Iso^{i+1}_{x,y}(M)\to\Iso^i_{x,y}(M)$. We
say that $x\in M$ is \emph{$k$-regular} if there exists an open
neighborhood $U_x$ of $x$ such that
$\Iso^{\germ}_{x,y}(M)\to\Iso^k_{x,y}(M)$ is surjective for every
$y\in U_x$.

If $v$ is a $C^\varepsilon$ local vector field on $M$ defined around
$x$, we denote the $i$-jet of $v$ at $x$ by $j_x^iv$. Let
$\Vect^i_x(M)$ denote the space of all such $i$-jets at $x$. A
(local) $C^\varepsilon$ vector field $v$ on $M$ is called a
(\emph{local}) \emph{$C^\varepsilon$ Killing field} if the local
flow generated by $v$ consists of local $C^\varepsilon$ isometries.
Note that there are natural maps
$\Kill(M)\to\Kill^{\germ}_x(M)\to\Vect^i_x(M)$, where $\Kill(M)$ and
$\Kill^{\germ}_x(M)$ are the Lie algebras of $C^\varepsilon$ Killing
fields and germs at $x$ of local $C^\varepsilon$ Killing fields
defined around $x$, respectively. We denote the kernel of
$\ev_x:\Kill^\germ_x(M)\to T_xM$ by $\Kill^\germ_x(M)_0$, which is a
subalgebra of $\Kill^\germ_x(M)$.

The $C^\varepsilon$ geometric structure $\sigma$ is \emph{$i$-rigid}
if $\Iso^{i+1}_x(M)\to\Iso^i_x(M)$ is injective for every $x\in M$,
and is \emph{rigid} if it is $i$-rigid for some $i\ge r$. Rigid
geometric structures have the following properties.

\begin{theorem}[Gromov]\label{T:Gromov1}
Let $M$ be a connected $C^\varepsilon$ manifold, and let $\sigma$ be
an $i$-rigid $C^\varepsilon$ geometric structure of order $r$ on
$M$, where $i\ge r>0$. Then we have
\begin{itemize}
\item[(1)] $\sigma$ is $i'$-rigid for any $i'\ge i$.
\item[(2)] For any $x\in M$, $\Iso^{\germ}_x(M)\to\Iso^{i}_x(M)$ is
injective.
\item[(3)] For any $x\in M$, $\Kill(M)\to\Kill^{\germ}_x(M)$ and
$\Kill^{\germ}_x(M)\to\Vect^{i}_x(M)$ are injective. In particular,
$\Kill(M)$ and $\Kill^{\germ}_x(M)$ are finite-dimensional.
\item[(4)] $\Iso(M)$ is a Lie group and its action on $M$ is
$C^\varepsilon$. Moreover, for any $x\in M$, $\Iso^{\germ}_x(M)$ is
a Lie group.
\item[(5)] There exist
an $\Iso^\germ(M)$-invariant open dense subset $M_\reg$ of $M$ and
an integer $k>i$ such that every point in $M_\reg$ is $k$-regular.
\end{itemize}
\end{theorem}

\begin{proof}
Detailed proof of (1) can be found in \cite{CQ03, Fe}. (2) is
implied from \cite[1.2.A and 1.5.A]{Gr}. We sketch a proof of (3)
using (2). If $v$ is a local Killing field defined around $x$ with
$j^i_xv=0$, and if $f_t$ is the local flow generated by $v$, then by
\cite[Lem. 2.1]{CQ03}, we have $j^i_xf_t=j^i_x\id_M$ for every $t$.
Thus (2) implies that $f_t=\id_M$ on some neighborhood $U_t$ of $x$.
Hence $\langle v\rangle_x=0$. This proves the injectivity of
$\Kill^{\germ}_x(M)\to\Vect^{i}_x(M)$. Now if $w\in\Kill(M)$ with
$\langle w\rangle_x=0$, then the set $Z=\{y\in M\mid j^i_yw=0\}$ is
closed and nonempty. By the injectivity of
$\Kill^{\germ}_y(M)\to\Vect^{i}_y(M)$, $Z$ is also open. So $Z=M$,
and hence $w=0$. This proves the injectivity of
$\Kill(M)\to\Kill^{\germ}_x(M)$. (4) is proved in \cite[Sect.
1.6.H]{Gr}. (5) is proved in \cite[Thm. 1.6.F]{Gr} and
\cite{Be,CQ03,Fe,Ze00,Ze02}.
\end{proof}

Note that $\LL(\Iso(M))$ can be identified with the Lie algebra of
complete Killing fields on $M$ via the induced infinitesimal action.
Under this identification, if $v\in\Kill(M)$ is complete, then
$t\mt\exp(-tv)$ is the flow generated by $v$. Similarly, for $x\in
M$, $\LL(\Iso^{\germ}_x(M))$ can be identified with
$\Kill^\germ_x(M)_0$ in such a way that for a local Killing field
$v$ which is defined around $x$ and vanishes at $x$, if $f_t$ is the
local flow generated by $v$, then $\exp(-t\langle
v\rangle_x)=\langle f_t\rangle_x$.

The next theorem holds only for analytic structures.

\begin{theorem}[Gromov]\label{T:Gromov2}
Let $M$ be a connected analytic manifold with an $i$-rigid analytic
$\A$-structure.
\begin{itemize}
\item[(1)] If $M$ is compact, then there exists
$k>i$ such that every point in $M$ is $k$-regular.
\item[(2)] If $M$ is simply connected, then for any $x\in M$,
$\Kill(M)\to\Kill^{\germ}_x(M)$ is bijective.
\end{itemize}
\end{theorem}

\begin{proof}
(1) is proved in \cite[Sect. 1.7.B]{Gr}, and (2) is proved in
\cite[Sect. 1.7]{Gr}.
\end{proof}

Finally, we remark that a geometric structure $\sigma:F^r(M)\to V$
naturally induces a geometric structure
$\tilde{\sigma}:F^r(\widetilde{M})\to V$ on $\widetilde{M}$ by
$\tilde{\sigma}(\tilde{\beta})=\sigma(j^r_{\tilde{x}}\pi\circ\tilde{\beta})$,
where $\tilde{x}\in\widetilde{M}$ and $\tilde{\beta}\in
F^r(\widetilde{M})_{\tilde{x}}$. If $\sigma$ is a rigid
$C^\varepsilon$ $\A$-structure, then so is $\tilde{\sigma}$. If a
connected Lie group $G$ acts on $M$ by $C^\varepsilon$ isometries,
then the induced $C^\varepsilon$ action of $\widetilde{G}$ on
$\widetilde{M}$ is also isometric.

\section{Zariski hulls of fundamental groups}\label{S:Zariski}

Let $M$ be a connected $C^\varepsilon$ manifold, and let
$\Gamma=\pi_1(M)$. We denote by $\Diff^\germ(M)$ the groupoid of
germs of local $C^\varepsilon$ diffeomorphisms of $M$, by
$\Diff_{x,y}^\germ(M)$ the subset of $\Diff^\germ(M)$ consisting of
germs with source $x$ and target $y$, by $\Vect(M)$ the Lie algebra
of $C^\varepsilon$ vector fields on $M$, and by $\Vect_x^\germ(M)$
the Lie algebra of germs at $x$ of local $C^\varepsilon$ vector
fields defined around $x$. For $\tilde{x}\in\widetilde{M}$, we
denote by
\begin{equation}\label{E:lambda}
\lambda_{\tilde{x}}:\Vect(\widetilde{M})\to\Vect^\germ_x(M)
\end{equation}
the composition of the natural homomorphism
$\Vect(\widetilde{M})\to\Vect^\germ_{\tilde{x}}(\widetilde{M})$ and
the natural isomorphism
$\Vect^\germ_{\tilde{x}}(\widetilde{M})\to\Vect^\germ_x(M)$, where
$x=\pi(\tilde{x})$. In other words, for $w\in \Vect(\widetilde{M})$,
we have $\lambda_{\tilde{x}}(w)=d\pi(\langle w\rangle_{\tilde{x}})$.
The deck transformations induce a representation of $\Gamma$ in
$\Vect(\widetilde{M})$ by taking differentials. If
$\cW\subset\Vect(\widetilde{M})$ is a finite-dimensional
$\Gamma$-invariant subspace, we denote the restricted representation
in $\cW$ by $\rho_\cW:\Gamma\to\GL(\cW)$.

\begin{definition}\label{D:hull}
Let $\cW$ be a finite-dimensional $\Gamma$-invariant subspace of
$\Vect(\widetilde{M})$.
\begin{itemize}
\item[(1)] Let $x,y\in M$, $\sA_{x,y}\subset\Diff_{x,y}^\germ(M)$. Choose $\tilde{x}\in\pi^{-1}(x)$,
$\tilde{y}\in\pi^{-1}(y)$. The \emph{Zariski hull of $\Gamma$ in
$\sA_{x,y}$ relative to $\cW$} is defined as
\begin{align*}
\Hull_\Gamma^\cW(\sA_{x,y})=\{&\varphi\in\sA_{x,y}\mid\text{there
exists }
A\in\overline{\overline{\rho_\cW(\Gamma)}} \text{ such that }\\
&d\varphi\circ\lambda_{\tilde{x}}|_{\cW}=\lambda_{\tilde{y}}\circ
A\}.
\end{align*}
\item[(2)] Let $\sA$ be a subgroupoid of $\Diff^\germ(M)$. The
\emph{Zariski hull of $\Gamma$ in $\sA$ relative to $\cW$} is
defined as
$$\Hull_\Gamma^\cW(\sA)=\bigcup_{x,y\in
M}\Hull_\Gamma^\cW(\sA\cap\Diff_{x,y}^\germ(M)).$$
\item[(3)] Let $x\in M$, $\La_x<\Vect_x^\germ(M)$. Choose $\tilde{x}\in\pi^{-1}(x)$. The
\emph{Zariski hull of $\Gamma$ in $\La_x$ relative to $\cW$} is
defined as
\begin{align*}
\Hull_\Gamma^\cW(\La_x)=\{&\eta\in\La_x\mid\text{there exists }
B\in\LL\left(\overline{\overline{\rho_\cW(\Gamma)}}\right) \text{
such that }\\
&\ad(\eta)\circ\lambda_{\tilde{x}}|_{\cW}=\lambda_{\tilde{x}}\circ
B\},
\end{align*}
where we view $\ad(\eta):\Vect_x^\germ(M)\to\Vect_x^\germ(M)$.
\end{itemize}
\end{definition}

It is obvious that if $\cW=0$, then $\Hull_\Gamma^\cW(\sA)=\sA$ and
$\Hull_\Gamma^\cW(\La_x)=\La_x$. But to justify Definition
\ref{D:hull} for general $\cW$, we need the following lemma.

\begin{lemma}\label{L:lambda}
For $\tilde{x}\in\widetilde{M}$ and $\gamma\in\Gamma$, we have
$\lambda_{\gamma\tilde{x}}=\lambda_{\tilde{x}}\circ d\gamma^{-1}$.
\end{lemma}

\begin{proof}
Since $\pi=\pi\circ\gamma^{-1}$, we have
$$
\lambda_{\gamma\tilde{x}}(w)=d\pi(\langle
w\rangle_{\gamma\tilde{x}})=d\pi(d\gamma^{-1}(\langle
w\rangle_{\gamma\tilde{x}}))=d\pi(\langle
d\gamma^{-1}(w)\rangle_{\tilde{x}})=\lambda_{\tilde{x}}(d\gamma^{-1}(w))
$$
for all $w\in\Vect(\widetilde{M})$.
\end{proof}

\begin{lemma}\label{L:hull0}
Under the situation of Definition \ref{D:hull}, we have
\begin{itemize}
\item[(1)] $\Hull_\Gamma^\cW(\sA_{x,y})$ is
independent of the choices of $\tilde{x}$ and $\tilde{y}$, and hence
$\Hull_\Gamma^\cW(\sA)$ is well-defined.
\item[(2)] $\Hull_\Gamma^\cW(\La_x)$ is
independent of the choice of $\tilde{x}$.
\end{itemize}
\end{lemma}

\begin{proof}
(1) If $\tilde{x}'\in\pi^{-1}(x)$ and $\tilde{y}'\in\pi^{-1}(y)$,
then there exist $\gamma_1,\gamma_2\in\Gamma$ such that
$\tilde{x}'=\gamma_1\tilde{x}$ and $\tilde{y}'=\gamma_2\tilde{y}$.
Let $\varphi\in\sA_{x,y}$. Suppose that there exists
$A\in\overline{\overline{\rho_\cW(\Gamma)}}$ such that
$d\varphi\circ\lambda_{\tilde{x}}|_{\cW}=\lambda_{\tilde{y}}\circ
A$. Then by Lemma \ref{L:lambda}, we have
\begin{align*}
d\varphi\circ\lambda_{\tilde{x}'}|_{\cW}&=d\varphi\circ\lambda_{\tilde{x}}\circ\rho_\cW(\gamma_1)^{-1}
=\lambda_{\tilde{y}}\circ
A\rho_\cW(\gamma_1)^{-1}\\
&=\lambda_{\tilde{y}'}\circ\rho_\cW(\gamma_2)A\rho_\cW(\gamma_1)^{-1}.
\end{align*}
Since
$\rho_\cW(\gamma_2)A\rho_\cW(\gamma_1)^{-1}\in\overline{\overline{\rho_\cW(\Gamma)}}$,
this proves (1).

(2) Let $\gamma\in\Gamma$ and $\eta\in\La_x$. If there exists
$B\in\LL\left(\overline{\overline{\rho_\cW(\Gamma)}}\right)$ such
that
$\ad(\eta)\circ\lambda_{\tilde{x}}|_{\cW}=\lambda_{\tilde{x}}\circ
B$, then by Lemma \ref{L:lambda}, we have
\begin{align*}
\ad(\eta)\circ\lambda_{\gamma\tilde{x}}|_{\cW}&=\ad(\eta)\circ\lambda_{\tilde{x}}\circ\rho_\cW(\gamma)^{-1}
=\lambda_{\tilde{x}}\circ
B\rho_\cW(\gamma)^{-1}\\
&=\lambda_{\gamma\tilde{x}}\circ\rho_\cW(\gamma)B\rho_\cW(\gamma)^{-1}.
\end{align*}
Since
$\rho_\cW(\gamma)B\rho_\cW(\gamma)^{-1}\in\LL\left(\overline{\overline{\rho_\cW(\Gamma)}}\right)$,
this proves (2).
\end{proof}

\begin{lemma}\label{L:hull}
\begin{itemize}
\item[(1)] Under the situation of Definition \ref{D:hull}(2), $\Hull_\Gamma^\cW(\sA)$ is a subgroupoid of $\sA$.
\item[(2)] Under the situation of Definition \ref{D:hull}(3), $\Hull_\Gamma^\cW(\La_x)$ is
a subalgebra of $\La_x$.
\end{itemize}
\end{lemma}

\begin{proof}
(1) It suffices to prove that if
$\varphi_1\in\Hull_\Gamma^\cW(\sA_{x,y})$ and
$\varphi_2\in\Hull_\Gamma^\cW(\sA_{x,z})$, then
$\varphi_2\circ\varphi_1^{-1}\in\Hull_\Gamma^\cW(\sA_{y,z})$. Let
$\tilde{x}\in\pi^{-1}(x)$, $\tilde{y}\in\pi^{-1}(y)$,
$\tilde{z}\in\pi^{-1}(z)$, and let
$A_1,A_2\in\overline{\overline{\rho_\cW(\Gamma)}}$ be such that
$d\varphi_1\circ\lambda_{\tilde{x}}|_{\cW}=\lambda_{\tilde{y}}\circ
A_1$ and
$d\varphi_2\circ\lambda_{\tilde{x}}|_{\cW}=\lambda_{\tilde{z}}\circ
A_2$. Then
$$d(\varphi_2\circ\varphi_1^{-1})\circ\lambda_{\tilde{y}}|_{\cW}
=d\varphi_2\circ\lambda_{\tilde{x}}\circ A_1^{-1}=\lambda_{\tilde{z}}\circ
A_2A_1^{-1}.$$ Since
$A_2A_1^{-1}\in\overline{\overline{\rho_\cW(\Gamma)}}$, we have
$\varphi_2\circ\varphi_1^{-1}\in\Hull_\Gamma^\cW(\sA_{y,z})$.

(2) Let $\eta_i\in\Hull_\Gamma^\cW(\La_x)$ $(i=1,2)$,
$\tilde{x}\in\pi^{-1}(x)$, and let
$B_i\in\LL\left(\overline{\overline{\rho_\cW(\Gamma)}}\right)$ be
such that
$\ad(\eta_i)\circ\lambda_{\tilde{x}}|_{\cW}=\lambda_{\tilde{x}}\circ
B_i$. Then
\begin{align*}
\ad([\eta_1,\eta_2])\circ\lambda_{\tilde{x}}|_{\cW}
&=\ad(\eta_1)\circ\ad(\eta_2)\circ\lambda_{\tilde{x}}|_{\cW}-\ad(\eta_2)\circ\ad(\eta_1)\circ\lambda_{\tilde{x}}|_{\cW}\\
&=\ad(\eta_1)\circ\lambda_{\tilde{x}}\circ
B_2-\ad(\eta_2)\circ\lambda_{\tilde{x}}\circ B_1\\
&=\lambda_{\tilde{x}}\circ B_1B_2-\lambda_{\tilde{x}}\circ B_2B_1\\
&=\lambda_{\tilde{x}}\circ [B_1,B_2].
\end{align*}
Thus $[\eta_1,\eta_2]\in\Hull_\Gamma^\cW(\La_x)$.
\end{proof}

Now we assume that $M$ is endowed with a rigid $C^\varepsilon$
geometric structure and $\cW\subset\Kill(\widetilde{M})$. For $x\in
M$ and $\La_x<\Kill_x^\germ(M)$, we denote
\begin{equation}\label{E:a_x^W}
\La_x^\cW=\{\eta\in\La_x\mid[\eta,\lambda_{\tilde{x}}(\cW)]=0\},
\end{equation}
where $\tilde{x}\in\pi^{-1}(x)$. Note that by Lemma \ref{L:lambda},
$\lambda_{\tilde{x}}(\cW)$ is independent of the choice of
$\tilde{x}$. The following lemma gives a direct relation between
$\Hull_\Gamma^\cW(\La_x)$ and
$\overline{\overline{\rho_\cW(\Gamma)}}$.

\begin{lemma}\label{L:exact}
Let $M$ be a connected $C^\varepsilon$ manifold with a rigid
$C^\varepsilon$ geometric structure, $\cW$ be a $\Gamma$-invariant
subspace of $\Kill(\widetilde{M})$, $x\in M$, and
$\La_x<\Kill_x^\germ(M)$. Then there exists an exact sequence
$$0\to
\La_x^\cW\hookrightarrow\Hull_\Gamma^\cW(\La_x)\to\LL\left(\overline{\overline{\rho_\cW(\Gamma)}}\right)$$
of Lie algebra homomorphisms.
\end{lemma}

\begin{proof}
Let $\tilde{x}\in\pi^{-1}(x)$, and denote
$\cW_x=\lambda_{\tilde{x}}(\cW)$. By Theorem \ref{T:Gromov1}(3), the
restriction of $\lambda_{\tilde{x}}$ to $\cW$ is injective. Hence it
induces an isomorphism $\lambda_{\tilde{x}}':\cW\to\cW_x$. Let
$\La'_x=\{\eta\in\La_x\mid\ad(\eta)(\cW_x)\subset\cW_x\}$. For
$\eta\in\La'_x$, we denote
$\delta(\eta)=\lambda_{\tilde{x}}'^{-1}\circ\ad(\eta)\circ\lambda_{\tilde{x}}'$.
Then $\delta:\La'_x\to\Lgl(\cW)$ is a representation and
$\ker(\delta)=\La_x^\cW$. It is easy to see that
$\Hull_\Gamma^\cW(\La_x)=\delta^{-1}\left(\LL\left(\overline{\overline{\rho_\cW(\Gamma)}}\right)\right)$.
Thus $\delta$ restricts to a homomorphism
$\Hull_\Gamma^\cW(\La_x)\to\LL\left(\overline{\overline{\rho_\cW(\Gamma)}}\right)$
with kernel $\La_x^\cW$. This completes the proof.
\end{proof}

Recall that there is a natural identification
$\LL(\Iso^{\germ}_x(M))=\Kill^\germ_x(M)_0$ such that if $v$ is a
local Killing field defined around $x$ with $\ev_x(v)=0$ and $f_t$
is the local flow generated by $v$, then $\exp(-t\langle
v\rangle_x)=\langle f_t\rangle_x$. The next result will be used in
the proof of Theorem \ref{T:main}(2).

\begin{lemma}\label{L:L-hull}
Let $M$, $\cW$, and $x$ be as in Lemma \ref{L:exact}, and let
$\sA_x$ be a closed subgroup of $\Iso^\germ_x(M)$. Then
$\Hull_\Gamma^\cW(\sA_{x})$ is a closed subgroup of $\sA_x$ and
$\LL(\Hull_\Gamma^\cW(\sA_{x}))=\Hull_\Gamma^\cW(\LL(\sA_x))$.
\end{lemma}

\begin{proof}
Let $\tilde{x}$, $\cW_x$ and $\lambda_{\tilde{x}}'$ be as in the
proof of Lemma \ref{L:exact}, and let $\sA'_x=\{\varphi\in\sA_x\mid
d\varphi(\cW_x)=\cW_x\}$. Then $\sA'_x$ is a closed subgroup of
$\sA_x$. For $\varphi\in\sA'_x$, we denote
$\rho(\varphi)=\lambda_{\tilde{x}}'^{-1}\circ
d\varphi\circ\lambda_{\tilde{x}}'$. Then $\rho:\sA'_x\to\GL(\cW)$ is
a representation and
$\Hull_\Gamma^\cW(\sA_x)=\rho^{-1}\left(\overline{\overline{\rho_\cW(\Gamma)}}\right)$.
Thus $\Hull_\Gamma^\cW(\sA_{x})$ is a closed subgroup of $\sA_x$.
Denote $\La_x=\LL(\sA_x)$, and let $\La'_x$ and $\delta$ be as in
the proof of Lemma \ref{L:exact}. By taking Lie derivatives, we see
that $\LL(\sA'_x)=\La'_x$ and $d\rho=\delta$. Thus
$\LL(\Hull_\Gamma^\cW(\sA_{x}))=\delta^{-1}\left(\LL\left(\overline{\overline{\rho_\cW(\Gamma)}}\right)\right)
=\Hull_\Gamma^\cW(\La_x)$.
\end{proof}

\section{Discompact groups}\label{S:discompact}

Let $\KK$ be a locally compact non-discrete field. The notion of
$\KK$-discompact groups was introduced by Shalom \cite{Sh} to
distinguish those $\KK$-groups for which the Borel density theorem
holds. A $\KK$-group $\bG$ is \emph{$\KK$-compact} if $\bG_\KK$ is
compact, and is \emph{$\KK$-discompact} if any $\KK$-homomorphism
from $\bG$ to a $\KK$-compact $\KK$-group is trivial. Obviously,
$\KK$-discompact groups are connected. It was proved in \cite{Sh}
that any $\KK$-group $\bG$ admits a unique maximal $\KK$-discompact
$\KK$-subgroup $R_d(\bG)$ of $\bG$, called the
\emph{$\KK$-discompact radical} of $\bG$, which is characteristic
for all $\KK$-automorphisms of $\bG$, such that $\bG/R_d(\bG)$ is
$\KK$-compact.

If $\mathrm{char}\,\KK=0$, the situation becomes much simpler.
Recall that a normal $\KK$-subgroup $\bN$ of a $\KK$-group $\bG$ is
\emph{$\KK$-cocompact} if $\bG_\KK/\bN_\KK$ is compact. If
$\mathrm{char}\,\KK=0$, then $\bN$ is $\KK$-cocompact if and only if
$\bG/\bN$ is $\KK$-compact (\cite[Sect. 2.5]{Sh}). Hence $\bG$ is
$\KK$-discompact if and only if it does not admit proper normal
$\KK$-cocompact $\KK$-subgroups. This enables us to characterize the
$\KK$-discompactness of $\bG$ as follows.

\begin{proposition}\label{P:discompact}
Let $\bG$ be a connected $\KK$-group, where $\mathrm{char}\,\KK=0$.
Then $\bG$ is $\KK$-discompact if and only if the radical $R(\bG)$
is $\KK$-split and all almost $\KK$-simple factors of the semisimple
$\KK$-group $\bG/R(\bG)$ are $\KK$-isotropic.
\end{proposition}

\begin{proof}
It follows from \cite[Thm. 3.6 and Prop. 3.12]{Sh} and \cite[Thm.
3.1]{PR}.
\end{proof}

Let $G=\bG_\RR$ be a real algebraic group, where $\bG$ is an
$\RR$-group. By taking the Zariski closure of $G$ in $\bG$, we may
assume without loss of generality that $G$ is Zariski dense in
$\bG$. We say that $G$ is \emph{discompact} if $\bG$ is
$\RR$-discompact. Then $G$ is discompact if and only if it has no
proper normal cocompact algebraic subgroups. Note that if $G$ is
discompact then it is Zariski connected. As in \cite{BFM}, we refer
to $R_d(G)=R_d(\bG)_\RR$ as the \emph{discompact radical} of $G$,
which is the unique maximal discompact algebraic subgroups of $G$
and is characteristic (for algebraic automorphisms of $G$) and
cocompact in $G$. It is obvious that $R(G)=(R(\bG)_\RR)_0$, where
$R(G)$ is the Lie-theoretic radical of $G$. We say that $R(G)$ is
\emph{$\RR$-split} if so is $R(\bG)$.

\begin{corollary}\label{C:discompact}
Let $G$ be a Zariski connected real algebraic group. Then $G$ is
discompact if and only if $R(G)$ is $\RR$-split and $G_0/R(G)$ has
no compact factors.
\end{corollary}

\begin{proof}
It follows from Proposition \ref{P:discompact} and the fact that
$G_0/R(G)$ is locally isomorphic to $(\bG/R(\bG))_\RR$.
\end{proof}

In the rest of this section, we study the discompactness of a Lie
group $G$ with respect to a (finite-dimensional) real representation
$\rho:G\to\GL(\cV)$. We will make extensive use of Levi subgroups
of $G$, which are maximal connected semisimple Lie subgroups of $G$
and satisfy the following properties.
\begin{itemize}
\item All Levi subgroups of $G$ are mutually conjugate and are locally
isomorphic to $G_0/R(G)$. If $L$ is one of them, then $G_0=L\cdot
R(G)$ (\cite{Va84}).
\item If $L$ is a connected semisimple Lie
subgroup of $G$, $N$ is a connected solvable normal Lie subgroup of
$G$, and $G_0=L\cdot N$, then $L$ is a Levi subgroup of $G$ and
$N=R(G)$ (\cite[Lem. 4.9.1]{Ad}).
\item If $\varphi:G\to H$ is a homomorphism of Lie
groups and $L$ is a Levi subgroup of $G$, then $\varphi(L)$ is a
Levi subgroup of $\varphi(G)$ and $R(\varphi(G))=\varphi(R(G))$
(\cite[Lem. 5.5.1]{Ad}).
\end{itemize}

\begin{proposition}\label{P:rho-discompact}
Let $G$ be a connected Lie group, and let $\rho:G\to\GL(\cV)$ be a
representation of $G$ in a real vector space $\cV$. Then $G$ is
$\rho$-discompact if and only if
\begin{itemize}
\item[(1)] all elements in $\rho(R(G))$ have only real eigenvalues, and
\item[(2)] any Levi subgroup of $\rho(G)$ has no compact factors.
\end{itemize}
\end{proposition}

\begin{proof}
Since $\rho(R(G))=R(\rho(G))$, we may assume without loss of
generality that $G$ is a connected Lie subgroup of $\GL(\cV)$ and
$\rho$ is the inclusion map. Let $L$ be a Levi subgroup of $G$. Then
$G=L\cdot R(G)$. Since $\left(\overline{\overline{L}}\right)_0=L$,
we have
$\left(\overline{\overline{G}}\right)_0=L\cdot\left(\overline{\overline{R(G)}}\right)_0$
(see \cite[Prop. 4.6.4]{Mo}). This implies that (i) $L$ is a Levi
subgroup of $\overline{\overline{G}}$, and (ii)
$R\left(\overline{\overline{G}}\right)=\left(\overline{\overline{R(G)}}\right)_0$.
Assertion (i) implies that $L$ is locally isomorphic to
$\left(\overline{\overline{G}}\right)_0\Big/R\left(\overline{\overline{G}}\right)$.
Hence $L$ has no compact factors if and only if so does
$\left(\overline{\overline{G}}\right)_0\Big/R\left(\overline{\overline{G}}\right)$.
We also have
\begin{align*}
&\text{all eigenvalues of every $g\in R(G)$ are real}\\
\Longleftrightarrow&\text{$R(G)$ fixes a full flag of $\cV$} &&\text{(by \cite[Cor. 1.29]{Kn})}\\
\Longleftrightarrow&\text{$\left(\overline{\overline{R(G)}}\right)_0$ fixes a full flag of $\cV$}\\
\Longleftrightarrow&\text{$R\left(\overline{\overline{G}}\right)$ fixes a full flag of $\cV$} &&\text{(by assertion (ii))}\\
\Longleftrightarrow&\text{$R\left(\overline{\overline{G}}\right)$ is
$\RR$-split} &&\text{(by \cite[Thm. 15.4(iii)]{Bo91}).}
\end{align*}
Now the proposition follows from Corollary \ref{C:discompact}.
\end{proof}

The discompactness of $G$ with respect to its adjoint representation
is of particular interest to us (this corresponds to the case of
Lemma \ref{L:discompact} where $G$ acts faithfully on $M$ and
$\cV=\cG$). Proposition \ref{P:rho-discompact} implies the following
characterization of the $\Ad$-discompactness of $G$.

\begin{corollary}\label{C:Ad-discompact}
Let $G$ be a connected Lie group. Then $G$ is $\Ad$-discompact if
and only if $R(G)$ is split solvable and $G/R(G)$ has no compact
factors.
\end{corollary}

\begin{proof}
By Proposition \ref{P:rho-discompact}, $G$ is $\Ad$-discompact if
and only if (i) all eigenvalues of $\Ad(g)$ are real for every $g\in
R(G)$, and (ii) any Levi subgroup of $\Ad(G)$ has no compact
factors. Since $R(G)$ is normal in $G$, any subspace of $\Lg$
containing $\LL(R(G))$ is $R(G)$-invariant. So for $g\in R(G)$, all
eigenvalues of $\Ad(g)$ are real if and only if all eigenvalues of
$\Ad(g)|_{\LL(R(G))}$ are real. Hence assertion (i) holds if and
only if $R(G)$ is split solvable. On the other hand, any Levi
subgroup of $\Ad(G)$ is locally isomorphic to $G/R(G)$. So assertion
(ii) holds if and only if $G/R(G)$ has no compact factors. Hence the
corollary follows.
\end{proof}

Now we use Corollary \ref{C:Ad-discompact} to prove Lemma
\ref{L:discompact}. We need the following lemma.

\begin{lemma}\label{L:subrep}
Let $\rho:G\to\GL(\cV)$ be a representation of a Lie group $G$ in a
real vector space $\cV$, let $\cV_0\subset\cV$ be a $G$-invariant
subspace, and let $\rho_1:G\to\GL(\cV_0)$,
$\rho_2:G\to\GL(\cV/\cV_0)$ be the subrepresentation and quotient
representation, respectively.
\begin{itemize}
\item[(1)] If $G$ is $\rho$-discompact, then it is $\rho_1$-discompact and $\rho_2$-discompact.
\item[(2)] If $G$ is connected, then the converse of (1) also holds.
\end{itemize}
\end{lemma}

\begin{proof}
Denote $\GL(\cV,\cV_0)=\{A\in\GL(\cV)\mid A(\cV_0)=\cV_0\}$, and
consider the algebraic homomorphisms
$\varphi_1:\GL(\cV,\cV_0)\to\GL(\cV_0)$ and
$\varphi_2:\GL(\cV,\cV_0)\to\GL(\cV/\cV_0)$ defined by
$\varphi_1(A)=A|_{\cV_0}$, $\varphi_2(A)(v+\cV_0)=Av+\cV_0$. It is
easy to see that $\rho(G)\subset\GL(\cV,\cV_0)$ and
$\rho_i(G)=\varphi_i(\rho(G))$, $i=1,2$.

(1) Let $N$ be a normal cocompact algebraic subgroup of
$\overline{\overline{\rho_i(G)}}$, $i=1$ or $2$. By \cite[Cor.
4.6.5]{Mo}, $\varphi_i\left(\overline{\overline{\rho(G)}}\right)$ is
an open (and Zariski dense) subgroup of
$\overline{\overline{\varphi_i(\rho(G))}}=\overline{\overline{\rho_i(G)}}$.
So $N\cap\varphi_i\left(\overline{\overline{\rho(G)}}\right)$ is
cocompact in $\varphi_i\left(\overline{\overline{\rho(G)}}\right)$,
and hence $\varphi_i^{-1}(N)\cap\overline{\overline{\rho(G)}}$ is
cocompact in $\overline{\overline{\rho(G)}}$. Since
$\overline{\overline{\rho(G)}}$ is discompact, we have
$\overline{\overline{\rho(G)}}\subset\varphi_i^{-1}(N)$. Thus
$\varphi_i\left(\overline{\overline{\rho(G)}}\right)\subset N$. But
$\varphi_i\left(\overline{\overline{\rho(G)}}\right)$ is Zariski
dense in $\overline{\overline{\rho_i(G)}}$. So
$N=\overline{\overline{\rho_i(G)}}$. Hence
$\overline{\overline{\rho_i(G)}}$ is discompact, i.e., $G$ is
$\rho_i$-discompact.

(2) Suppose that $G$ is connected and $\rho_i$-discompact, $i=1,2$.
Then Proposition \ref{P:rho-discompact} implies that all eigenvalues
of $\rho_i(g)$ are real for every $g\in R(G)$. Hence all eigenvalues
of $\rho(g)$ are real for every $g\in R(G)$. By Proposition
\ref{P:rho-discompact}, to prove that $G$ is $\rho$-discompact, it
suffices to show that any Levi subgroup $L$ of $\rho(G)$ has no
compact factors. Suppose on the contrary that $L$ has a compact
simple factor $L'$. Since $\varphi_i(L)$ is a Levi subgroup of
$\varphi_i(\rho(G))=\rho_i(G)$, Proposition \ref{P:rho-discompact}
implies that $\varphi_i(L)$ has no compact factors. So
$L'\subset\ker(\varphi_1)\cap\ker(\varphi_2)$. But
$\ker(\varphi_1)\cap\ker(\varphi_2)$ is unipotent, a contradiction.
This completes the proof.
\end{proof}

\begin{remark}
Lemma \ref{L:subrep}(2) does not hold without the connectedness
condition. For example, let $G=\left\{\begin{pmatrix}
a & 0\\
0 & b
\end{pmatrix}\Big| a,b\in\RR\backslash\{0\}, a^2=b^2\right\}$,
let $\rho$ be the natural representation of $G$ in $\cV=\RR^2$, and
let $\cV_0=\RR\times\{0\}$. Since $G$ is a real algebraic group and
is not Zariski connected, it is not $\rho$-discompact. But $G$ is
discompact with respect to the subrepresentation in $\cV_0$ and the
quotient representation in $\cV/\cV_0$.
\end{remark}

\begin{proof}[Proof of Lemma \ref{L:discompact}]
(1) Denote the quotient representation of $G$ in $\cV/\cG$ by
$\rho_{\cV/\cG}$. By Lemma \ref{L:subrep}(2), it suffices to prove
that $G$ is discompact with respect to $\rho_\cG$ and
$\rho_{\cV/\cG}$. Since the induced infinitesimal action of $\Lg$ on
$M$ is $G$-equivariant, $\rho_\cG$ is equivalent to a quotient
representation of the adjoint representation of $G$. By Corollary
\ref{C:Ad-discompact} and Lemma \ref{L:subrep}(1), $G$ is
$\rho_\cG$-discompact. On the other hand, the conditions in Lemma
\ref{L:discompact}(1) imply that all eigenvalues of
$\rho_{\cV/\cG}(g)$ are real for every $g\in R(G)$, and that any
Levi subgroup of $\rho_{\cV/\cG}(G)$ has no compact factors. By
Proposition \ref{P:rho-discompact}, $G$ is
$\rho_{\cV/\cG}$-discompact.

(2) Suppose that $G$ is $\rho_\cV$-discompact and the $G$-action is
faithful. By Lemma \ref{L:subrep}(1), $G$ is $\rho_\cG$-discompact
and $\rho_{\cV/\cG}$-discompact. The faithfulness of the $G$-action
implies that the induced infinitesimal action of $\Lg$ on $M$ is
injective. So $\rho_\cG$ is equivalent to the adjoint representation
of $G$. Hence $G$ is $\Ad$-discompact. By Corollary
\ref{C:Ad-discompact}, $R(G)$ is split solvable and $G/R(G)$ has no
compact factors. On the other hand, since $G$ is
$\rho_{\cV/\cG}$-discompact, Proposition \ref{P:rho-discompact}
implies that all eigenvalues of $\rho_{\cV/\cG}(g)$ are real for
every $g\in R(G)$. By \cite[Cor. 1.29]{Kn}, $R(G)$ fixes a full flag
of $\cV/\cG$, hence fixes a sequence of subspaces of $\cV$ with the
required properties.
\end{proof}

We conclude this section with the following result.

\begin{proposition}\label{P:discompact-radical}
Let $G$ be a Lie group, and let $\rho:G\to\GL(\cV)$ be a
representation of $G$ in a real vector space $\cV$. Then there
exists a unique maximal connected $\rho$-discompact Lie subgroup $N$
of $G$. Moreover, $N$ is a closed normal subgroup of $G$ and $G/N$
is locally isomorphic to a compact Lie group.
\end{proposition}

\begin{proof}
Denote $D=R_d\left(\overline{\overline{\rho(G_0)}}\right)$. We prove
that the group $N=\rho^{-1}(D)_0$ satisfies the requirement. We
first show that $N$ is $\rho$-discompact. Let $L$ be a Levi subgroup
of $D$. By Corollary \ref{C:discompact}, $R(D)$ is $\RR$-split and
$L$ has no compact factors. By \cite[Cor. 4.6.7]{Mo}, we have
$L=(L,L)\subset\left(\overline{\overline{\rho(G_0)}},\overline{\overline{\rho(G_0)}}\right)\subset\rho(G_0)$.
Thus
\begin{equation}\label{E:discompact-radical}
L\subset(\rho(G)\cap D)_0=\rho(\rho^{-1}(D))_0=\rho(N).
\end{equation}
Note that $\rho(N)\subset D_0$. So $D_0=L\cdot R(D)=\rho(N)\cdot
R(D)$. This implies that the quotient homomorphism $\varphi:D_0\to
D_0/R(D)$ maps $\rho(N)$ onto $D_0/R(D)$. So the connected solvable
Lie subgroup $\varphi(R(\rho(N)))$ of $D_0/R(D)$ is normal, hence
must be trivial. Thus
$\rho(R(N))=R(\rho(N))\subset\ker(\varphi)=R(D)$. Since $R(D)$ is
$\RR$-split, all elements in $\rho(R(N))$ have only real
eigenvalues. On the other hand, \eqref{E:discompact-radical} also
implies that $L$ is a Levi subgroup of $\rho(N)$. Hence any Levi
subgroup of $\rho(N)$ has no compact factors. By Proposition
\ref{P:rho-discompact}, $N$ is $\rho$-discompact.

Now we prove that $N$ contains every connected $\rho$-discompact Lie
subgroup $H$ of $G$. Since
$\overline{\overline{\rho(H)}}\subset\overline{\overline{\rho(G_0)}}$
is discompact, we have $\overline{\overline{\rho(H)}}\subset D$. So
$H\subset\rho^{-1}(D)_0=N$. Hence $N$ is the unique maximal
connected $\rho$-discompact Lie subgroup of $G$.

Obviously, $N$ is closed in $G$. For any $g\in G$,
$\overline{\overline{\rho(gNg^{-1})}}=\rho(g)\overline{\overline{\rho(N)}}\rho(g)^{-1}$
is discompact. So $gNg^{-1}$ is $\rho$-discompact, and hence is
contained in $N$. Thus $N$ is normal in $G$. Consider the natural
homomorphisms
$$G_0/N\to G_0/(\rho^{-1}(D)\cap G_0)\to
\rho(G_0)/(D\cap\rho(G_0))\to\overline{\overline{\rho(G_0)}}/D$$ of
Lie groups. The first is a covering homomorphism, the second is an
isomorphism, and the third is injective. So $\LL(G_0/N)$ is
isomorphic to a subalgebra of
$\LL\left(\overline{\overline{\rho(G_0)}}/D\right)$. Since
$\overline{\overline{\rho(G_0)}}/D$ is compact, $\LL(G_0/N)$ is
compact. Thus $G/N$ is locally isomorphic to a compact Lie group.
This completes the proof.
\end{proof}

We refer to the group $N$ in Proposition \ref{P:discompact-radical}
as the \emph{discompact radical of $G$ with respect to $\rho$} (or
the \emph{$\rho$-discompact radical} of $G$), and denote it by
$R^\rho_d(G)$.

\begin{remark}
\begin{itemize}
\item[(1)] The quotient group $G/R^\rho_d(G)$ may be noncompact.
Here is an example for the adjoint representation. Let
$G=\RR\ltimes\CC^2$, where $\RR$ acts on $\CC^2$ by
$t.(z_1,z_2)=(e^{i\alpha t}z_1,e^{i\beta t}z_2)$, here
$\alpha,\beta\in\RR$ are linearly independent over $\QQ$. It is easy
to see that $R^\Ad_d(G)=\{0\}\times\CC^2$. So
$G/R^\Ad_d(G)\cong\RR$.
\item[(2)] Similar to the proof of Corollary \ref{C:Ad-discompact},
one can show that for any Lie group $G$, $R^\Ad_d(G)$ is the unique
maximal member among all the connected normal Lie subgroups $N$ of
$G$ such that $R(N)$ is split solvable and $N/R(N)$ has no compact
factors. Since we do not need this fact in this paper, we leave the
details to the reader.
\end{itemize}
\end{remark}

\section{Borel density theorem for algebraic
quotients}\label{S:Borel}

The original form of the Borel density theorem asserts that any
lattice in a real semisimple algebraic group without compact factors is Zariski dense
(\cite{Bo60}). This was extended by several authors and played a
crucial role in the ergodic theory of group actions (see
\cite{Sh,Zi84} and the references therein). One version of the
theorem that we will need below is as follows (\cite[Thms. 3.9 and
3.3]{Sh}).

\begin{theorem}\label{T:Shalom}
Let $\KK$ be a locally compact non-discrete field, $\bG$ be a
$\KK$-discompact $\KK$-group, and $\bK$ be a $\KK$-subgroup of
$\bG$. Suppose that there exists a $\bG_\KK$-invariant finite Borel
measure on $\bG_\KK/\bK_\KK$. Then $\bK=\bG$.
\end{theorem}

In what follows we assume that $\mathrm{char}\,\KK=0$, and consider
certain action of $\bG_\KK$ on an algebraic quotient
$\bV_\KK/\bH_\KK$, where $\bG$, $\bH$ are $\KK$-groups and $\bV$ is
a $\KK$-variety. More precisely, let $\bG\times\bH$ act
$\KK$-regularly $\bV$. Then $\bG_\KK$ acts naturally on
$\bV_\KK/\bH_\KK$. We endow $\bV_\KK/\bH_\KK$ with the quotient
topology and quotient Borel structure. Then the $\bG_\KK$-action on
$\bV_\KK/\bH_\KK$ is both continuous and measurable.

Recall that a Borel space is \emph{countably separated} if there
exists a countable family of Borel sets which separates points, and
is \emph{standard} if it is isomorphic to a Borel subset of a
complete separate metric space. Obviously, standard Borel spaces are
countably separated. We first list some basic properties of the
topology and Borel structure on $\bV_\KK/\bH_\KK$.

\begin{lemma}\label{L:quotient1}
\begin{itemize}
\item[(1)] For every $x\in\bV_\KK/\bH_\KK$, $\{x\}$ is locally
closed.
\item[(2)] $\bV_\KK/\bH_\KK$ is a standard
Borel space.
\item[(3)] The quotient Borel structure on $\bV_\KK/\bH_\KK$
coincides with the Borel structure generated by the quotient
topology on $\bV_\KK/\bH_\KK$.
\end{itemize}
\end{lemma}

\begin{proof}
(1) follows from the local closedness of the $\bH_\KK$-orbits in
$\bV_\KK$. (2) follows from \cite[Thm. 2.9]{Ef} (see also \cite[Thm.
A.7]{Zi84}). (3) follows from \cite[Thm. 2.6]{Ef}.
\end{proof}

Recall that a measurable action of an lcsc group $G$ on a countably
separated Borel space $X$ is \emph{tame} if the quotient Borel
structure on $X/G$ is countably separated. It is well-known the
$(\bG_\KK\times\bH_\KK)$-action on $\bV_\KK$ is tame.

\begin{lemma}\label{L:quotient2}
\begin{itemize}
\item[(1)] The $\bG_\KK$-action on $\bV_\KK/\bH_\KK$ is tame.
\item[(2)] For every $x\in\bV_\KK/\bH_\KK$, there exists a
$\KK$-subgroup $\bK$ of $\bG$ such that $\bG_\KK(x)$ is an open
subgroup of finite index in $\bK_\KK$.
\end{itemize}
\end{lemma}

\begin{proof}
(1) Since $(\bV_\KK/\bH_\KK)/\bG_\KK$ is isomorphic to
$\bV_\KK/(\bG_\KK\times\bH_\KK)$ as Borel spaces, the tameness of
the $\bG_\KK$-action on $\bV_\KK/\bH_\KK$ follows from the tameness
of the $(\bG_\KK\times\bH_\KK)$-action on $\bV_\KK$.

(2) Let $v\in x$, and denote $\bS=(\bG\times\bH)(v)$. Since
$\mathrm{char}\,\KK=0$, $\bS$ is a $\KK$-subgroup of $\bG\times\bH$.
Let $p_1:\bG\times\bH\to\bG$ be the projection. It is easy to see
that $\bG_\KK(x)=p_1(\bS_\KK)$. Let $\bK=p_1(\bS)$. Then $\bK$ is a
$\KK$-subgroup of $\bG$, and $p_1(\bS_\KK)$ is an open subgroup of
finite index in $\bK_\KK$ (see \cite[p. 113, Cor. 1 and p. 319, Cor.
2]{PR}). Hence the result follows.
\end{proof}

Now we prove the Borel density theorem for algebraic quotients
(Theorem \ref{T:Borel}).

\begin{theorem}\label{T:Borelrestate}
Let $\KK$ be a locally compact non-discrete field of characteristic
$0$, $\bG$ and $\bH$ be $\KK$-groups such that $\bG$ is
$\KK$-discompact, and $\bV$ be a $\KK$-variety such that
$\bV_\KK\ne\emptyset$. Suppose that $\bG\times\bH$ acts
$\KK$-regularly on $\bV$. Let $\mu$ be a finite Borel measure on
$\bV_\KK/\bH_\KK$ such that either
\begin{itemize}
\item[(1)] $\mu$ is $\bG_\KK$-invariant, or
\item[(2)] $\KK=\RR$ and $\mu$ is invariant under a Zariski dense Lie subgroup $G$ of $\bG_\RR$.
\end{itemize}
Then for $\mu$-a.e. $x\in\bV_\KK/\bH_\KK$, $\bG_\KK(x)$ is an open
subgroup of finite index in $\bG_\KK$.
\end{theorem}

\begin{proof}
If $\KK\ne\RR$, we let $G=\bG_\KK$. Since $\bV_\KK/\bH_\KK$ is a
standard Borel $G$-space and $G$ is lcsc (if $\KK=\RR$ we consider
the Lie group topology on $G$), by \cite[Thm. 4.4]{Va63}, the
$G$-invariant measure $\mu$ admits an ergodic decomposition. This
means that there exists a map $x\mt\beta_x$ from $\bV_\KK/\bH_\KK$
to the space of $G$-invariant ergodic probability measures on
$\bV_\KK/\bH_\KK$ such that for any Borel set
$E\subset\bV_\KK/\bH_\KK$, the function $x\mt\beta_x(E)$ is
measurable and we have $\mu(E)=\int_X\beta_x(E)d\mu(x)$. So we may
assume without loss of generality that $\mu$ is $G$-ergodic. Since
the $\bG_\KK$-action on $\bV_\KK/\bH_\KK$ is tame, by \cite[Lem.
2.2]{Sh}, $\mu$ is supported on a single $\bG_\KK$-orbit, say
$\mathcal{O}\subset\bV_\KK/\bH_\KK$. It suffices to prove that for
every $x\in\mathcal{O}$, $\bG_\KK(x)$ is an open subgroup of finite
index in $\bG_\KK$.

Let $x\in\mathcal{O}$. By Lemma \ref{L:quotient2}(2), there exists a
$\KK$-subgroup $\bK$ of $\bG$ such that $\bG_\KK(x)$ is an open
subgroup of finite index in $\bK_\KK$. In particular, $\bG_\KK(x)$
is closed in $\bG_\KK$. So $\bG_\KK/\bG_\KK(x)$ is a standard Borel
space. By \cite[Thm. A.4]{Zi84}, $\bG_\KK/\bG_\KK(x)$ is isomorphic
to $\mathcal{O}$ as Borel $\bG_\KK$-spaces. So we may view $\mu$ as
a $G$-invariant finite Borel measure on $\bG_\KK/\bG_\KK(x)$. Let
$p:\bG_\KK/\bG_\KK(x)\to\bG_\KK/\bK_\KK$ be the natural projection.
Then $p_*(\mu)$ is a $G$-invariant finite Borel measure on
$\bG_\KK/\bK_\KK$. We claim that $p_*(\mu)$ is $\bG_\KK$-invariant.
Indeed, if $\KK\ne\RR$, this has been assumed as a condition. If
$\KK=\RR$, since $G$ is Zariski dense in $\bG_\RR$, the
$\bG_\RR$-invariance of $p_*(\mu)$ follows from \cite[Cor. 2.6]{Da}.
Now by Theorem \ref{T:Shalom}, we have $\bK=\bG$. Thus $\bG_\KK(x)$
is an open subgroup of finite index in $\bG_\KK$. This completes the
proof.
\end{proof}

\begin{remark}\label{R:counterexample}
Under the conditions of Theorem \ref{T:Borelrestate}, even if
$\KK=\RR$, $\bG_\RR$ may have no fixed point in $\bV_\RR/\bH_\RR$.
For example, let $\bG=\bH=\bV=\CC\setminus\{0\}$ be the
$1$-dimensional $\RR$-split torus, and let $\bG\times\bH$ act on
$\bV$ by $(g,h).v=gh^2v$. Then $\bV_\RR/\bH_\RR$ consists of two
points, each of them has stabilizer $(0,+\infty)$, which is not an
algebraic subgroup of $\bG_\RR$. Recall that by Rosenlicht's theorem
\cite{Ro}, $\bV_\RR/\bH_\RR$ is a finite disjoint union
$\bigcup_iV_i$ of real algebraic varieties. The above example shows
that even if each $V_i$ can be chosen to be $\bG_\RR$-invariant, the
$\bG_\RR$-action on $V_i$ may fail to be algebraic (for otherwise
the stabilizers would be algebraic subgroups of $\bG_\RR$).
\end{remark}

For the convenience of later applications, we propose the following
corollary of Theorem \ref{T:Borelrestate}.

\begin{corollary}\label{C:Borel}
Let $G$ be a connected Lie group, $\rho:G\to\GL(\cV)$ be a
representation of $G$ in a real vector space $\cV$ such that $G$ is
$\rho$-discompact, $H$ be a real algebraic group, and $V$ be a real
algebraic variety. Suppose that $\GL(\cV)\times H$ acts regularly on
$V$, and that the induced $G$-action on $V/H$ preserves a finite
Borel measure $\mu$ on $V/H$. Then $(V/H)^G$ is a $\mu$-conull
$G$-saturated constructible subset of $V/H$.
\end{corollary}

\begin{proof}
Theorem \ref{T:Borelrestate} implies that $\mu$-a.e. point in $V/H$
is fixed by some open subgroup of $\overline{\overline{\rho(G)}}$.
But an open subgroup of $\overline{\overline{\rho(G)}}$ must contain
$\rho(G)$. So $(V/H)^G$ is $\mu$-conull. By Lemmas
\ref{L:quotient1}(1) and \ref{L:saturated}, for every $x\in(V/H)^G$,
$\{x\}$ is $G$-saturated. So Lemma \ref{L:boolean} implies that
$(V/H)^G$ is $G$-saturated. The assertion that $(V/H)^G$ is
constructible follows from a standard inductive argument where we
apply Lemma \ref{L:Fconstructible} to the set of nonsingular points
in $V$.
\end{proof}

\section{Proof of Theorem \ref{T:main}}\label{S:proof}

In this section we prove Theorem \ref{T:main}. For convenience, we
restate the theorem as follows.

\begin{theorem}\label{T:mainrestate}
Let $M$ be a connected $n$-dimensional $C^\varepsilon$ manifold with
an $i$-rigid $C^\varepsilon$ $\A$-structure $\sigma$ of order $r$
and type $V$, where $i\ge r>0$, and let a connected Lie group $G$
act on $M$ by $C^\varepsilon$ isometries. Let $\mu$ be a finite
Borel measure on $M$, $\cV$ be a $G$-invariant subspace of
$\Kill(M)$, and $\cW$ be a $\Gamma$-invariant subspace of
$\Kill(\widetilde{M})^\cG$, where $\Gamma=\pi_1(M)$. Suppose that
$G$ is $\rho_\cV$-discompact and $\ker(\rho_\cV)\cdot G_\mu=G$. Then
there exists a $G$-saturated constructible set $M'\subset M_\reg$
with $\mu(M_\reg\backslash M')=0$ such that for every $x\in M'$, we
have
\begin{itemize}
\item[(1)] the $\Hull_\Gamma^\cW(\Iso^\germ(M)^\cV)$-orbit of $x$
contains a $G$-invariant open dense subset of $\overline{Gx}$, and
\item[(2)] $T_xGx\subset\ev_x(\Hull_\Gamma^\cW(\Kill^\germ_x(M)^\cV))$.
\end{itemize}
\end{theorem}

The proof of the theorem generalizes ideas in the proofs of Gromov's centralizer theorem in \cite{Gr,Zi93}.
We make use of two geometric structures
associated with the $i$-jets of Killing fields in $\cV$ and $\cW$,
respectively. Let $k>i$ be an integer. We first introduce a
notation. Let $N_1$ and $N_2$ be $n$-dimensional $C^\varepsilon$
manifolds, and let $x\in N_1$, $y\in N_2$. If $f$ is a local
$C^\varepsilon$ diffeomorphism from $N_1$ to $N_2$ defined around
$x$ sending $x$ to $y$ and $\xi\in\Vect^\germ_x(N_1)$, then we
denote $(j_x^kf)_*(j_x^i\xi)=j_y^idf(\xi)$. Since $k>i$,
$(j_x^kf)_*(j_x^i\xi)$ depends only on $j_x^kf$ and $j_x^i\xi$. This
defines a linear isomorphism
$(j_x^kf)_*:\Vect^i_x(N_1)\to\Vect^i_y(N_2)$.

Now we define the geometric structure induced by $\cV$. This is
motivated by \cite[Sect. 5.1]{Gr}. Let $L(\cV,\Vect^i_0(\RR^n))$
denote the space of linear maps from $\cV$ to $\Vect^i_0(\RR^n)$.
Then $\GL^k(n)\times\GL(\cV)$ acts regularly on
$L(\cV,\Vect^i_0(\RR^n))$ by
$$(\alpha,A).\ell=\alpha_*\circ\ell\circ A^{-1}, \quad
\alpha\in\GL^k(n), A\in\GL(\cV), \ell\in L(\cV,\Vect^i_0(\RR^n)).$$
We define a map
$$\tau_{\cV,k,i}:F^k(M)\to L(\cV,\Vect^i_0(\RR^n))$$ by
$$\tau_{\cV,k,i}(\beta)(v)=\beta^{-1}_*(j_x^iv), \quad x\in M, \beta\in F^k(M)_x,
v\in \cV.$$ Note that $G$ acts on $F^k(M)$ by
$g.\beta=j_x^kg\circ\beta$, where $\beta\in F^k(M)_x$.

\begin{lemma}\label{L:equivariance}
The map $\tau_{\cV,k,i}$ is a $C^\varepsilon$ $\A$-structure, and is
$G$-equivariant with respect to $\rho_\cV$.
\end{lemma}

\begin{proof}
It suffices to check the $\GL^k(n)$-equivariance and
$G$-equivariance of $\tau_{\cV,k,i}$. Let $\alpha\in\GL^k(n)$, $g\in
G$, $x\in M$, $\beta\in F^k(M)_x$. Then for $v\in \cV$ we have
$$\tau_{\cV,k,i}(\beta\circ\alpha^{-1})(v)
=(\alpha\circ\beta^{-1})_*(j_x^iv)
=(\alpha_*\circ\tau_{\cV,k,i}(\beta))(v)=(\alpha.\tau_{\cV,k,i}(\beta))(v)$$
and
\begin{align*}
\tau_{\cV,k,i}(g.\beta)(v)&=\tau_{\cV,k,i}(j^k_xg\circ\beta)(v)
=(\beta^{-1}\circ j^k_{gx}(g^{-1}))_*(j_{gx}^iv)\\
&=\beta^{-1}_*(j^i_x dg^{-1}(v))=\beta^{-1}_*(j^i_x\rho_\cV(g^{-1})v)\\
&=(\tau_{\cV,k,i}(\beta)\circ\rho_\cV(g)^{-1})(v)=(\rho_\cV(g).\tau_{\cV,k,i}(\beta))(v).
\end{align*}
Hence
$\tau_{\cV,k,i}(\beta\circ\alpha^{-1})=\alpha.\tau_{\cV,k,i}(\beta)$
and $\tau_{\cV,k,i}(g.\beta)=\rho_\cV(g).\tau_{\cV,k,i}(\beta)$.
\end{proof}

Next we define the geometric structure induced by $\cW$. By applying
the above construction to $\widetilde{M}$, we get a $C^\varepsilon$
$\A$-structure $\tau_{\cW,k,i}:F^k(\widetilde{M})\to
L(\cW,\Vect^i_0(\RR^n))$. An argument similar to the proof of the
second assertion of Lemma \ref{L:equivariance} shows that
$\tau_{\cW,k,i}$ is $\Gamma$-equivariant with respect to the
representation $\rho_\cW:\Gamma\to\GL(\cW)$. The map
$F^k(\widetilde{M})\to F^k(M)$ sending $\tilde{\beta}\in
F^k(\widetilde{M})_{\tilde{x}}$ to
$j_{\tilde{x}}^k\pi\circ\tilde{\beta}$ induces an identification
$F^k(\widetilde{M})/\Gamma\cong F^k(M)$. Thus $\tau_{\cW,k,i}$
induces a continuous map
$$\upsilon_{\cW,k,i}:F^k(M)\to
L(\cW,\Vect^i_0(\RR^n))\Big/\overline{\overline{\rho_\cW(\Gamma)}},$$
where the algebraic quotient
$L(\cW,\Vect^i_0(\RR^n))\Big/\overline{\overline{\rho_\cW(\Gamma)}}$
is endowed with the quotient topology. Writing explicitly, we have
$$\upsilon_{\cW,k,i}(\beta)=\left(\overline{\overline{\rho_\cW(\Gamma)}}\right)
(\tau_{\cW,k,i}((j^{k}_{\tilde{x}}\pi)^{-1}\circ\beta)), \quad
x\in M, \beta\in F^k(M)_x, \tilde{x}\in\pi^{-1}(x),$$ where
$$\tau_{\cW,k,i}((j^{k}_{\tilde{x}}\pi)^{-1}\circ\beta)(w)=\beta^{-1}_*(d\pi(j_{\tilde{x}}^iw)),
\quad w\in\cW.$$ Note that $\GL^k(n)$ acts continuously on
$L(\cW,\Vect^i_0(\RR^n))\Big/\overline{\overline{\rho_\cW(\Gamma)}}$.

\begin{lemma}\label{L:invariance}
The map $\upsilon_{\cW,k,i}$ is a continuous $G$-invariant geometric
structure.
\end{lemma}

\begin{proof}
The $\GL^k(n)$-equivariance of $\upsilon_{\cW,k,i}$ follows from
that of $\tau_{\cW,k,i}$. To prove the $G$-invariance, we consider
the induced action of $\widetilde{G}$ on $\widetilde{M}$. Since
$\cW\subset\Kill(\widetilde{M})^\cG$, $\cW$ is
$\widetilde{G}$-invariant and the representation
$\widetilde{G}\to\GL(\cW)$ is trivial. An argument similar to the
proof of Lemma \ref{L:equivariance} shows that $\tau_{\cW,k,i}$ is
$\widetilde{G}$-invariant. By passing to the quotients, we obtain
the $G$-invariance of $\upsilon_{\cW,k,i}$.
\end{proof}

Now we are prepared to prove Theorem \ref{T:mainrestate}. We first
prove assertion (1).

\begin{proof}[Proof of Theorem \ref{T:mainrestate}(1)]
Let $k>i$ be such that every point in $M_\reg$ is $k$-regular.
Consider the $(k-r)$-th prolongation $\sigma^{k-r}:F^k(M)\to
J_n^{k-r}(V)$ of $\sigma$ and the geometric structures
$\tau_{\cV,k,i}$ and $\upsilon_{\cW,k,i}$ constructed above. Since
the $G$-action on $M$ is isometric, $\sigma^{k-r}$ is $G$-invariant.
By Lemmas \ref{L:equivariance} and \ref{L:invariance},
$\tau_{\cV,k,i}$ is $G$-equivariant with respect to $\rho_\cV$, and
$\upsilon_{\cW,k,i}$ is $G$-invariant. Consider the regular action
of $\GL^k(n)\times\GL(\cV)\times\GL(\cW)$ on
$$W=J_n^{k-r}(V)\times L(\cV,\Vect^i_0(\RR^n))\times
L(\cW,\Vect^i_0(\RR^n))$$ defined by
$$(\alpha,A_1,A_2).(\zeta,\ell_1,\ell_2)=(\alpha.\zeta,\alpha_*\circ\ell_1\circ A_1^{-1},\alpha_*\circ\ell_2\circ
A_2^{-1}),$$ where
$(\alpha,A_1,A_2)\in\GL^k(n)\times\GL(\cV)\times\GL(\cW)$ and
$(\zeta,\ell_1,\ell_2)\in W$. Then there is a natural identification
$$W\Big/\overline{\overline{\rho_\cW(\Gamma)}}=
J_n^{k-r}(V)\times L(\cV,\Vect^i_0(\RR^n))\times
\left(L(\cW,\Vect^i_0(\RR^n))\Big/\overline{\overline{\rho_\cW(\Gamma)}}\right).$$
The continuous geometric structure
$$\sigma^{k-r}\times\tau_{\cV,k,i}\times\upsilon_{\cW,k,i}:F^k(M)\to
W\Big/\overline{\overline{\rho_\cW(\Gamma)}}$$ is $G$-equivariant
with respect to $\rho_\cV$. So its continuous Gauss map
$$\theta:M\to
W\Big/\left(\GL^k(n)\times\overline{\overline{\rho_\cW(\Gamma)}}\right)$$
is also $G$-equivariant with respect to $\rho_\cV$. Since $\theta$
is measurable, it induces a finite Borel measure $\theta_*(\mu)$ on
$W\Big/\left(\GL^k(n)\times\overline{\overline{\rho_\cW(\Gamma)}}\right)$.
The condition $\ker(\rho_\cV)\cdot G_\mu=G$ implies that
$\theta_*(\mu)$ is $G$-invariant. Since $G$ is
$\rho_\cV$-discompact, by Corollary \ref{C:Borel}, the set
$$F=\left(W\Big/\left(\GL^k(n)\times\overline{\overline{\rho_\cW(\Gamma)}}\right)\right)
^G$$ is $G$-saturated, constructible, and $\theta_*(\mu)$-conull.
Let
$$M'=\theta^{-1}(F)\cap M_\reg.$$
Obviously, $\theta^{-1}(F)$ is constructible and $\mu$-conull. By
Lemma \ref{L:inverseimage}, $\theta^{-1}(F)$ is $G$-saturated. So
$M'$ is $G$-saturated and constructible, and satisfies
$\mu(M_\reg\backslash M')=0$. We prove that every $x\in M'$
satisfies the requirement of Theorem \ref{T:mainrestate}(1).

Let $p=\theta(x)$. Then $p$ is fixed by $G$. By Lemma
\ref{L:quotient1}(1), $\{p\}$ is locally closed. So
$\theta^{-1}(p)$, and hence $\theta^{-1}(p)\cap M_\reg$, is
$G$-invariant and locally closed. Let $C_x$ be the connected
component of $\theta^{-1}(p)\cap M_\reg$ containing $x$. We claim
that $C_x$ is also $G$-invariant and locally closed. Indeed, since
$C_x$ is closed in $\theta^{-1}(p)\cap M_\reg$, it is locally
closed. To see the $G$-invariance of $C_x$, let $y\in C_x$. Since
$Gy$ and $C_x$ are connected and $Gy\cap C_x\ne\emptyset$, $Gy\cup
C_x$ is also connected. But $Gy\cup C_x$ is contained in
$\theta^{-1}(p)\cap M_\reg$ and $C_x$ is a connected component of
$\theta^{-1}(p)\cap M_\reg$. So we must have $Gy\cup C_x=C_x$. Thus
$Gy\subset C_x$, and hence $C_x$ is $G$-invariant. Let
$S_x=C_x\cap\overline{Gx}$, which is obviously $G$-invariant. We
claim that $S_x$ is open dense in $\overline{Gx}$. Indeed, since
$C_x$ is open in $\overline{C_x}$, $S_x$ is open in
$\overline{C_x}\cap\overline{Gx}=\overline{Gx}$. On the other hand,
since $Gx\subset S_x$, $S_x$ is dense in $\overline{Gx}$. Thus it
suffices to prove that the
$\Hull_\Gamma^\cW(\Iso^\germ(M)^\cV)$-orbit of $x$ contains $S_x$.

We prove the stronger assertion that the
$\Hull_\Gamma^\cW(\Iso^\germ(M)^\cV)$-orbit of $x$ contains $C_x$.
Since $C_x$ is connected, we need only to show that for every
$\Hull_\Gamma^\cW(\Iso^\germ(M)^\cV)$-orbit $O$ in $M$, $O\cap C_x$
is open in $C_x$. If $O\cap C_x=\emptyset$, there is nothing to
prove. Otherwise, for any $y\in O\cap C_x$, since $y\in M_\reg$,
there exists an open neighborhood $U_y$ of $y$ in $M$ such that
$\Iso^{\germ}_{y,z}(M)\to\Iso^k_{y,z}(M)$ is surjective for every
$z\in U_y$. We will prove that $U_y\cap C_x\subset O$. If this is
true, then the open neighborhood $U_y\cap C_x$ of $y$ in $C_x$ will
be contained in $O\cap C_x$, and hence $O\cap C_x$ will be open in
$C_x$.

It remains to prove that $U_y\cap C_x\subset O$. Let $z\in U_y\cap
C_x$. Choose $\beta_y\in F^k(M)_y$ and $\beta_z\in F^k(M)_z$. Since
$y,z\in C_x$, we have $\theta(z)=\theta(y)$. So there exists
$\alpha\in\GL^k(n)$ such that
$$\alpha.(\sigma^{k-r}\times\tau_{\cV,k,i}\times\upsilon_{\cW,k,i})(\beta_y)
=(\sigma^{k-r}\times\tau_{\cV,k,i}\times\upsilon_{\cW,k,i})(\beta_z).$$
Hence we have
\begin{align}
\label{E:1} \alpha.\sigma^{k-r}(\beta_y)&=\sigma^{k-r}(\beta_z),\\
\label{E:2} \alpha.\tau_{\cV,k,i}(\beta_y)&=\tau_{\cV,k,i}(\beta_z),\\
\label{E:3}
\alpha.\upsilon_{\cW,k,i}(\beta_y)&=\upsilon_{\cW,k,i}(\beta_z).
\end{align}
Equation \eqref{E:1} implies that
$$\sigma^{k-r}((\beta_z\circ\alpha\circ\beta_y^{-1})\circ\beta_y)
=\alpha^{-1}.\sigma^{k-r}(\beta_z)=\sigma^{k-r}(\beta_y).$$
So $\beta_z\circ\alpha\circ\beta_y^{-1}\in\Iso^k_{y,z}(M)$. Since
$\Iso^{\germ}_{y,z}(M)\to\Iso^k_{y,z}(M)$ is surjective, there
exists $\varphi\in\Iso^{\germ}_{y,z}(M)$ such that
$j^k_y\varphi=\beta_z\circ\alpha\circ\beta_y^{-1}$. By \eqref{E:2},
for any $v\in \cV$ we have
\begin{align*}
j_z^id\varphi(\langle
v\rangle_y)&=(j_y^k\varphi)_*(j_y^iv)=(\beta_z\circ\alpha\circ\beta_y^{-1})_*(j_y^iv)
=(\beta_z)_*(\alpha.\tau_{\cV,k,i}(\beta_y)(v))\\
&=(\beta_z)_*(\tau_{\cV,k,i}(\beta_z)(v))
=(\beta_z)_*((\beta_z^{-1})_*(j_z^iv))=j_z^i(\langle v\rangle_z).
\end{align*}
Since $\Kill^{\germ}_z(M)\to\Vect^i_z(M)$ is injective, we have
$d\varphi(\langle v\rangle_y)=\langle v\rangle_z$ for all $v\in
\cV$. So $\varphi\in\Iso^\germ(M)^\cV$. Let
$\tilde{y}\in\pi^{-1}(y)$, $\tilde{z}\in\pi^{-1}(z)$,
$\beta_{\tilde{y}}=(j_{\tilde{y}}^k\pi)^{-1}\circ\beta_y$,
$\beta_{\tilde{z}}=(j_{\tilde{z}}^k\pi)^{-1}\circ\beta_z$. By
\eqref{E:3}, there exists
$A\in\overline{\overline{\rho_\cW(\Gamma)}}$ such that
$$\alpha.\tau_{\cW,k,i}(\beta_{\tilde{y}})=\tau_{\cW,k,i}(\beta_{\tilde{z}})\circ
A.$$ Let $\lambda_{\tilde{y}}$ and $\lambda_{\tilde{z}}$ be the
homomorphisms defined in \eqref{E:lambda}. Then for any $w\in \cW$
we have
\begin{align*}
j_z^id\varphi(\lambda_{\tilde{y}}(w))&=j_z^id\varphi(d\pi\langle
w\rangle_{\tilde{y}})=(j_y^k\varphi\circ
j_{\tilde{y}}^k\pi)_*(j_{\tilde{y}}^iw)
=(\beta_z\circ\alpha\circ\beta_y^{-1}\circ j_{\tilde{y}}^k\pi)_*(j_{\tilde{y}}^iw)\\
&=(\beta_z\circ\alpha\circ\beta_{\tilde{y}}^{-1})_*(j_{\tilde{y}}^iw)=(\beta_z)_*(\alpha.\tau_{\cW,k,i}(\beta_{\tilde{y}})(w))\\
&=(\beta_z)_*(\tau_{\cW,k,i}(\beta_{\tilde{z}})(Aw))=(\beta_z)_*((\beta_{\tilde{z}}^{-1})_*(j_{\tilde{z}}^iAw))\\
&=(j_{\tilde{z}}^k\pi)_*(j_{\tilde{z}}^iAw)=j_z^id\pi(\langle
Aw\rangle_{\tilde{z}})=j_z^i\lambda_{\tilde{z}}(Aw).
\end{align*}
By the injectivity of $\Kill^{\germ}_z(M)\to\Vect^i_z(M)$, we have
$d\varphi(\lambda_{\tilde{y}}(w))=\lambda_{\tilde{z}}(Aw)$ for all
$w\in\cW$. So $\varphi\in\Hull_\Gamma^\cW(\Iso^\germ(M)^\cV)$. Thus
$z$ lies in the $\Hull_\Gamma^\cW(\Iso^\germ(M)^\cV)$-orbit $O$ of
$y$. This proves that $U_y\cap C_x\subset O$, and hence completes
the proof of Theorem \ref{T:mainrestate}(1).
\end{proof}

Now we deduce the second assertion of the theorem from the first
one.

\begin{proof}[Proof of Theorem \ref{T:mainrestate}(2)]
Note that for $x\in M$, a $G$-invariant open dense subset of
$\overline{Gx}$ must contain $Gx$. So by Theorem
\ref{T:mainrestate}(1), it suffices to prove that for $x\in M$, if
the condition
\begin{equation}\label{E:condition}
\text{$Gx$ is contained in the
$\Hull_\Gamma^\cW(\Iso^\germ(M)^\cV)$-orbit of $x$}
\end{equation}
holds, then
$T_xGx\subset\ev_x(\Hull_\Gamma^\cW(\Kill^\germ_x(M)^\cV))$.

Let $\cV_x$ be the image of $\cV$ under the injection
$\Kill(M)\to\Kill_x^\germ(M)$. For $v\in\cV$, we denote
$\lambda_x(v)=\langle v\rangle_x$. Then we have an isomorphism
$\lambda_x:\cV\to\cV_x$. Let
$$\sA_x=\{\varphi\in\Iso_x^\germ(M)\mid
d\varphi(\cV_x)=\cV_x\}.$$ Then $\sA_x$ is a closed subgroup of
$\Iso_x^\germ(M)$. For $\varphi\in\sA_x$, we denote
$\rho(\varphi)=\lambda_x^{-1}\circ d\varphi\circ\lambda_x$. Then
$\rho:\sA_x\to\GL(\cV)$ is a representation. We first prove that if
\eqref{E:condition} holds, then
\begin{equation}\label{E:inclusion1}
\rho_\cV(G)\subset\rho(\Hull_\Gamma^\cW(\sA_x)).
\end{equation}
Let $g\in G$. Then \eqref{E:condition} implies that there exists
$\psi\in\Hull_\Gamma^\cW(\Iso^\germ(M)^\cV)$ with source $gx$ and
target $x$. Let $\varphi=\psi\circ\langle
g\rangle_x\in\Iso_x^\germ(M)$. We prove \eqref{E:inclusion1} by
showing that $\varphi\in\Hull_\Gamma^\cW(\sA_x)$ and
$\rho(\varphi)=\rho_\cV(g)$. Since $\psi$ centralizes $\cV$, for any
$v\in\cV$ we have
\begin{align*}
d\varphi\circ\lambda_x(v)&=d(\psi\circ
\langle g\rangle_x)(\langle v\rangle_x)=d\psi(\langle dg(v)\rangle_{gx})\\
&=\langle dg(v)\rangle_x=\lambda_x\circ\rho_\cV(g)(v).
\end{align*}
So $\varphi\in\sA_x$ and $\rho(\varphi)=\rho_\cV(g)$. Consider the
induced action of $\widetilde{G}$ on $\widetilde{M}$. Let
$\tilde{g}$ be an element in the preimage of $g$ under the covering
homomorphism $\widetilde{G}\to G$, and let
$\tilde{x}\in\pi^{-1}(x)$. Then $\pi(\tilde{g}\tilde{x})=gx$. By the
choice of $\psi$, there exists
$A\in\overline{\overline{\rho_\cW(\Gamma)}}$ such that
$d\psi(d\pi(\langle w\rangle_{\tilde{g}\tilde{x}}))=d\pi(\langle
Aw\rangle_{\tilde{x}})$ for all $w\in \cW$. Since
$\cW\subset\Kill(\widetilde{M})^\cG$, we have $d\tilde{g}(w)=w$ for
all $w\in \cW$. Thus
\begin{align*}
d\varphi(d\pi(\langle
w\rangle_{\tilde{x}}))&=d\psi(d(g\circ\pi)(\langle
w\rangle_{\tilde{x}}))=d\psi(d(\pi\circ\tilde{g})(\langle
w\rangle_{\tilde{x}}))\\
&=d\psi(d\pi(\langle
d\tilde{g}(w)\rangle_{\tilde{g}\tilde{x}}))=d\psi(d\pi(\langle
w\rangle_{\tilde{g}\tilde{x}}))=d\pi(\langle Aw\rangle_{\tilde{x}})
\end{align*}
for all $w\in \cW$. Hence $\varphi\in\Hull_\Gamma^\cW(\sA_x)$. This
proves \eqref{E:inclusion1}. Note that \eqref{E:inclusion1} and
Lemma \ref{L:L-hull} imply that
\begin{equation}\label{E:inclusion2}
d\rho_\cV(\Lg)=\LL(\rho_\cV(G))\subset\LL(\rho(\Hull_\Gamma^\cW(\sA_x)))=d\rho(\Hull_\Gamma^\cW(\LL(\sA_x))).
\end{equation}
Note also that by taking Lie derivatives, we have
$$\LL(\sA_x)=\{\xi\in\Kill_x^\germ(M)_0\mid[\xi,\cV_x]\subset\cV_x\},$$
and
$$d\rho_\cV(X)(v)=[\iota(X),v], \quad
d\rho(\xi)(v)=\lambda_x^{-1}([\xi,\langle v\rangle_x)]), \quad
X\in\Lg, \xi\in\LL(\sA_x), v\in\cV.$$

Now we prove that if \eqref{E:inclusion2} holds, then
$T_xGx\subset\ev_x(\Hull_\Gamma^\cW(\Kill^\germ_x(M)^\cV))$. Let
$X\in\Lg$. By \eqref{E:inclusion2}, there exists
$\xi\in\Hull_\Gamma^\cW(\LL(\sA_x))$ such that
$d\rho_\cV(X)=d\rho(\xi)$. This means that
$$[\langle\iota(X)\rangle_x,\langle
v\rangle_x]=\langle[\iota(X),v]\rangle_x=[\xi,\langle v\rangle_x]$$
for all $v\in\cV$. Let $\eta=\langle\iota(X)\rangle_x-\xi$. Then
$\eta\in\Kill^\germ_x(M)^\cV$. Moreover, since
$\xi\in\Hull_\Gamma^\cW(\LL(\sA_x))$ and
$\cW\subset\Kill(\widetilde{M})^\cG$, there exists
$B\in\LL\left(\overline{\overline{\rho_\cW(\Gamma)}}\right)$ such
that $$[\eta,d\pi(\langle w\rangle_{\tilde{x}})]=-[\xi,d\pi(\langle
w\rangle_{\tilde{x}})]=-d\pi(\langle Bw\rangle_{\tilde{x}})$$ for
all $w\in\cW$. Thus $\eta\in\Hull_\Gamma^\cW(\Kill^\germ_x(M)^\cV)$.
Note that $\ev_x(\xi)=0$. So $\ev_x(\eta)=\iota_x(X)$. Hence
$$T_xGx=\iota_x(\Lg)\subset\ev_x(\Hull_\Gamma^\cW(\Kill^\germ_x(M)^\cV)).$$
This proves Theorem \ref{T:mainrestate}(2).
\end{proof}

\section{Generalizations of Gromov's theorems}\label{S:proof1}

As explained in Section 1, Theorem \ref{T:main} unifies Gromov's centralizer, representation, and open dense orbit theorems. In this section we provide the details. We first prove Theorem \ref{T:cen-rep}, which generalizes the centralizer and representation theorems. Throughout this section we denote
$\Gamma=\pi_1(M)$.

\begin{proof}[Proof of Theorem \ref{T:cen-rep}]
Let $\cW=\Kill(\widetilde{M})^\cV$, which is obviously
$\Gamma$-invariant and contained in $\Kill(\widetilde{M})^\cG$. For
$x\in M$, let $\cG_x$ be the image of $\cG$ under the injection
$\Kill(M)\to\Kill^{\germ}_x(M)$, and denote
$$\La_x=\Kill^\germ_x(M)^\cV, \quad \Lh_x=\Hull_\Gamma^\cW(\La_x), \quad \Lf_x=\cG_x+\Lh_x.$$
Then $[\cG_x,\Lh_x]=0$. By Theorem \ref{T:main}(2), there exists a
$G$-saturated constructible set $M'\subset M_\reg$ with
$\mu(M_\reg\backslash M')=0$ such that
\begin{equation}\label{E:cen-rep-1}
T_xGx\subset\ev_x(\Lh_x), \quad x\in M'.
\end{equation}
We prove that every $x\in M'$ satisfies the requirements of Theorem
\ref{T:cen-rep}.

Firstly, since $\Lh_x<\La_x$, from \eqref{E:cen-rep-1} we obtain
$T_xGx\subset\ev_x(\La_x)$. On the other hand, since
$\ev_x(\cG_x)=T_xGx$, \eqref{E:cen-rep-1} implies that $\Lf_x$ is
equal to the sum of $\Lh_x$ and the kernel of the evaluation map
$\ev_x:\Lf_x\to T_xM$, which is a subalgebra of $\Lf_x$. By Lemma
\ref{L:subquotient3}, there exist $\Lz_x<\Lh_x$ and a Lie algebra
structure on $T_xGx$ such that $\ev_x(\Lz_x)=T_xGx$, and such that
$\ev_x:\cG_x\to T_xGx$ and $-\ev_x:\Lz_x\to T_xGx$ are Lie algebra
homomorphisms. In particular, we have $T_xGx\prec\Lh_x$. This also
implies that
$\iota_x=\ev_x\circ\langle\cdot\rangle_x\circ\iota:\Lg\to T_xGx$ is
a homomorphism. Thus $\Lg(x)=\ker(\iota_x)\lhd\Lg$, and we have
\begin{equation}\label{E:cen-rep-2}
\Lg/\Lg(x)\cong T_xGx\prec\Lh_x.
\end{equation}

If $\varepsilon=\omega$, then by Theorem \ref{T:Gromov2}(2), for any
$\tilde{x}\in\pi^{-1}(x)$, the homomorphism $\lambda_{\tilde{x}}$
defined in \eqref{E:lambda} restricts to an isomorphism from
$\Kill(\widetilde{M})$ onto $\Kill^\germ_x(M)$. This implies that
$\lambda_{\tilde{x}}(\cW)=\La_x$. So
$$d\pi(T_{\tilde{x}}\widetilde{G}\tilde{x})=T_xGx\subset\ev_x(\La_x)
=\ev_x(\lambda_{\tilde{x}}(\cW))=d\pi(\ev_{\tilde{x}}(\cW)).$$ Thus
$T_{\tilde{x}}\widetilde{G}\tilde{x}\subset\ev_{\tilde{x}}(\cW)$. On
the other hand, the fact $\lambda_{\tilde{x}}(\cW)=\La_x$ also
implies that the Lie algebra $\La_x^\cW$ defined in \eqref{E:a_x^W}
is equal to $Z(\La_x)$. Thus by Lemma \ref{L:exact}, we have
$\Lh_x/Z(\La_x)\prec\LL\left(\overline{\overline{\rho(\Gamma)}}\right)$.
Now from \eqref{E:cen-rep-2} and Lemma \ref{L:subquotient2} we
obtain
$$\ad(\Lg/\Lg(x))
\prec\ad(\Lh_x)\cong\Lh_x/Z(\Lh_x)\prec\Lh_x/Z(\La_x)\prec\LL\left(\overline{\overline{\rho(\Gamma)}}\right).$$
This completes the proof.
\end{proof}

\begin{remark}\label{R:locallyfree}
The assertion $\Lg(x)\lhd\Lg$ in Theorem \ref{T:cen-rep}(1) in fact
holds without the geometric structure. More precisely, if an
$\Ad$-discompact Lie group $G$ acts smoothly on a smooth manifold
$M$ and preserves a finite Borel measure $\mu$ on $M$, then there
exists a $G$-saturated constructible $\mu$-conull subset $M'\subset
M$ such that for every $x\in M$ we have $\Lg(x)\lhd\Lg$. This is
well-known if we only require $M'$ to be $\mu$-conull (see, e.g.,
\cite[proof of Cor. 3.6]{Zi93}). The topological requirements on
$M'$ can be easily obtained by taking \cite[Lem. 2.1]{St} into
account. This can be also proved directly by using the continuous
Gauss map of the geometric structure $\tau:F^1(M)\to L(\Lg,\RR^n)$
defined by
$\tau(\beta)(X)=\beta^{-1}\left(\left.\frac{d}{dt}\right|_{t=0}\exp(-tX)x\right)$,
where $\beta\in F^1(M)_x$ is viewed as an invertible linear map
$\RR^n\to T_xM$.
\end{remark}

The following result asserts that if $M$ is smooth and Gromov's representation
theorem does not hold, then the local Killing fields on $\widetilde{M}$ are highly non-extendable.

\begin{theorem}\label{T:rep-smooth}
Let $M$ be a connected smooth manifold with a rigid smooth
$\A$-structure and a finite smooth measure $\mu$, and let $G$ be a
connected noncompact simple Lie group with finite center. Suppose
that $G$ acts faithfully, smoothly, and isometrically on $M$ and
preserves $\mu$. Let
$\rho:\pi_1(M)\to\GL(\Kill(\widetilde{M})^{\Kill(M)})$ be the
representation induced by the deck transformations. Then at least
one of the following two assertions holds.
\begin{itemize}
\item[(1)]
$\overline{\overline{\rho(\pi_1(M))}}$ has a Lie subgroup locally
isomorphic to $G$.
\item[(2)] There exists a $G$-invariant open dense subset $U$ of
$M$, on which $G$ acts locally freely, such that for every $x\in U$,
$\Kill_x^\germ(M)$ has a subalgebra $\Lg_x$ satisfying the following
properties.
\begin{itemize}
\item[(i)] $\Lg_x$ centralizes $\Kill(M)$ and $\ev_x(\Lg_x)=T_xGx$.
\item[(ii)] The map $-\iota_x^{-1}\circ\ev_x:\Lg_x\to\Lg$ is a Lie algebra isomorphism.
\item[(iii)] For any
$\tilde{x}\in\pi^{-1}(x)$, every nonzero element in the image of
$\Lg_x$ under the natural isomorphism
$\Kill_x^\germ(M)\to\Kill_{\tilde{x}}^\germ(\widetilde{M})$ can not
be extended to a global Killing field on $\widetilde{M}$.
\end{itemize}
\end{itemize}
\end{theorem}

To prove Theorem \ref{T:rep-smooth}, we need the following
well-known result. We sketch the proof for completeness.

\begin{lemma}\label{L:locallyfree}
Let $M$ be a connected $C^\varepsilon$ manifold, and let $G$ be a
connected semisimple Lie group without compact factors. Suppose that
$G$ acts faithfully on $M$ by $C^\varepsilon$ diffeomorphisms and
preserves a finite smooth measure $\mu$ on $M$. Suppose also that
either $G$ has a finite center or $\varepsilon=\omega$. Then $G$
acts locally freely on a $G$-invariant open dense $\mu$-conull
subset of $M$.
\end{lemma}

\begin{proof}
The proof is similar to that for the simple group case given in
\cite{CQ03,Zi93}. Let $$U=\{x\in M\mid G(x) \text{ is discrete}\}.$$
Then $U$ is $G$-invariant and open. Since $\mu$ is smooth, it
suffices to prove that $U$ is $\mu$-conull. Let $G_i$ $(1\le i\le
k)$ be the simple factors of $G$, and let $M_i=M^{G_i}$. By Remark
\ref{R:locallyfree}, there exists a $\mu$-conull set $M'\subset M$
such that for every $x\in M'$, $G(x)_0$ is a normal subgroup of $G$,
hence is the product of some $G_i$. This implies that $M'\subset
U\cup\bigcup_{i=1}^kM_i$. Thus it suffices to show that
$\mu(M_i)=0$. If $G$ has a finite center, this follows from
\cite[Prop. 4]{Zi87}. If $\varepsilon=\omega$, then $M_i$ is a
proper analytic subset of $M$ and hence, by \cite[Cor. 3.6]{CQ03},
has measure zero.
\end{proof}

\begin{proof}[Proof of Theorem \ref{T:rep-smooth}]
Let $\cW=\Kill(\widetilde{M})^{\Kill(M)}$. For $x\in M$, let $\cG_x$
be the image of $\cG$ under the injection
$\Kill(M)\to\Kill^{\germ}_x(M)$, and denote
$$\La_x=\Kill^\germ_x(M)^{\Kill(M)}, \quad \Lh_x=\Hull_\Gamma^\cW(\La_x).$$
By Lemma \ref{L:discompact}(1), $G$ is discompact with respect to
the natural representation $G\to\GL(\Kill(M))$. Similar to the proof
of Theorem \ref{T:cen-rep}, we obtain a $G$-saturated constructible
set $M'\subset M_\reg$ with $\mu(M_\reg\backslash M')=0$ such that
for every $x\in M'$, there exist $\Lz_x<\Lh_x$ and a Lie algebra
structure on $T_xGx$ such that $\ev_x(\Lz_x)=T_xGx$, and such that
$\iota_x:\Lg\to T_xGx$ and $-\ev_x:\Lz_x\to T_xGx$ are Lie algebra
homomorphisms. Since $\mu$ is smooth, by Corollary
\ref{C:constructible}, $M'$ contains a $G$-invariant open dense
subset $U_1\subset M$. By Lemma \ref{L:locallyfree}, there exists a
$G$-invariant open dense subset $U_2\subset M$ such that $\Lg(x)=0$
for every $x\in U_2$. Thus for $x\in U=U_1\cap U_2$, $\iota_x$ is an
isomorphism of Lie algebras. In particular, $T_xGx$ is simple. By
using a Levi decomposition of $\Lz_x$, it is easy to see that there
exists $\Lg_x<\Lz_x$ such that $-\ev_x:\Lg_x\to T_xGx$ is an
isomorphism. Obviously, $\Lg_x$ satisfies (i) and (ii) in assertion
(2).

We prove the theorem by showing that either assertion (1) holds, or
for every $x\in U$, $\Lg_x$ satisfies (iii) in assertion (2). Note
that the latter means that if
\begin{equation}\label{E:rep-smooth}
x\in U, \quad \tilde{x}\in\pi^{-1}(x), \quad
w\in\Kill(\widetilde{M}), \quad \text{and} \quad
\lambda_{\tilde{x}}(w)\in\Lg_x,
\end{equation}
then $w=0$. For $x\in U$, by Lemma \ref{L:exact}, there exists a
homomorphism
$\Phi_x:\Lh_x\to\LL\left(\overline{\overline{\rho(\Gamma)}}\right)$
such that $\ker(\Phi_x)=\La_x^\cW$, where $\La_x^\cW$ is as in
\eqref{E:a_x^W}. Since $\ker(\Phi_x)\cap\Lg_x\lhd\Lg_x$, it is equal
to $0$ or $\Lg_x$. If there exists $x\in U$ such that
$\ker(\Phi_x)\cap\Lg_x=0$, then
$$\Lg\cong\Lg_x\cong\Phi_x(\Lg_x)<\LL\left(\overline{\overline{\rho(\Gamma)}}\right),$$
and hence we get assertion (1). Otherwise, for every $x\in U$, we
have $\ker(\Phi_x)\cap\Lg_x=\Lg_x$, and hence
$\Lg_x\subset\La_x^\cW$. In this case, \eqref{E:rep-smooth} implies
that $\lambda_{\tilde{x}}(w)\in\La_x$, hence centralizes the image
of $\Kill(M)\to\Kill_x^\germ(M)$. Since the restriction of
$\lambda_{\tilde{x}}$ to $\Kill(\widetilde{M})$ is injective, $w$
centralizes the image of $\Kill(M)\to\Kill(\widetilde{M})$. So
$w\in\cW$. This implies that
$[\lambda_{\tilde{x}}(w),\Lg_x]\subset[\lambda_{\tilde{x}}(w),\La_x^\cW]=0$.
Thus $\lambda_{\tilde{x}}(w)\in Z(\Lg_x)=0$, and hence $w=0$. This
completes the proof.
\end{proof}

A special case of Gromov's open dense orbit theorem is that if $G$ acts
topologically transitively on $M$, then $\Iso^\germ(M)$ has an open
dense orbit. The following consequence of Theorem \ref{T:main}(1)
asserts that under certain conditions, a subgroupoid of
$\Iso^\germ(M)$ already has an open dense orbit.

\begin{theorem}\label{T:opendense}
Let $M$ be a connected $C^\varepsilon$ manifold with a rigid
$C^\varepsilon$ $\A$-structure, $G$ be a connected Lie group which
acts on $M$ by $C^\varepsilon$ isometries, and $\cV$ be a
$G$-invariant subspace of $\Kill(M)$ such that $G$ is
$\rho_\cV$-discompact. Suppose that there exists a finite smooth
measure $\mu$ on $M$ such that $\ker(\rho_\cV)\cdot G_\mu=G$.
\begin{itemize}
\item[(1)] If the $G$-action is topologically transitive, then $\Iso^\germ(M)^\cV$ has an open dense
orbit.
\item[(2)] If the $G$-action is minimal, then $\Iso^\germ(M)^\cV$
is transitive on $M$.
\end{itemize}
\end{theorem}

\begin{proof}
By the $\cW=0$ case of Theorem \ref{T:main}(1) and Corollary
\ref{C:constructible}, there exists a $G$-invariant open dense
subset $U_0$ of $M$ such that for every $x\in U_0$, the
$\Iso^\germ(M)^\cV$-orbit of $x$ contains a $G$-invariant open dense
subset of $\overline{Gx}$. Let $x_0\in M$ be such that
$\overline{Gx_0}=M$. Then any $G$-invariant open subset of $M$ must
contains $x_0$. In particular, we have $x_0\in U_0$. Hence the
$\Iso^\germ(M)^\cV$-orbit $O$ of $x_0$ contains a $G$-invariant open
dense subset $U$ of $M$. Obviously, $O$ is dense in $M$. Thus to
prove Theorem \ref{T:opendense}(1), it suffices to show that $O$ is
open. Let $x\in O$. Then there exists $\varphi\in\Iso^\germ(M)^\cV$
with source $x_0$ and target $x$. Let $f$ be a local isometry
defined on an open neighborhood $U_f$ of $x_0$ such that $\langle
f\rangle_{x_0}=\varphi$. Since $x_0\in U$, we may assume that
$U_f\subset U\subset O$. Since $\varphi$ centralizes $\cV$, we have
$\langle df(v|_{U_f})\rangle_x=\langle v\rangle_x$ for all
$v\in\cV$. So there exists an open neighborhood $U'_f\subset U_f$ of
$x_0$ such that $df(v|_{U'_f})=v|_{f(U'_f)}$ for all $v\in\cV$.
Hence $\langle f\rangle_y\in\Iso^\germ(M)^\cV$ for all $y\in U'_f$.
This, together with the fact $U'_f\subset O$, implies that
$f(U'_f)\subset O$. Note that $f(U'_f)$ is a neighborhood of $x$. So
$O$ is open. This proves Theorem \ref{T:opendense}(1). If
furthermore the $G$-action is minimal, then as a $G$-invariant open
set, $U$ must be equal to $M$. Hence $O=M$. Thus $\Iso^\germ(M)^\cV$
is transitive on $M$. This proves Theorem \ref{T:opendense}(2).
\end{proof}

\begin{remark}\label{R:opendense}
\begin{itemize}
\item[(1)] If $\cV=0$, then $G$ is $\rho_\cV$-discompact and the condition
$\ker(\rho_\cV)\cdot G_\mu=G$ is satisfied for any finite smooth
measure $\mu$. Thus in this case Theorem \ref{T:opendense} reduces
to the above special case of Gromov's open dense orbit theorem.
\item[(2)] By Lemma \ref{L:discompact}, if $\mu$ is $G$-invariant, then the conditions of Theorem
\ref{T:opendense} are satisfied if $G$ is semisimple without compact
factors and $\cV=\Kill(M)$, or if $R(G)$ is split solvable, $G/R(G)$
has no compact factors, and $\cV=\cG$.
\end{itemize}
\end{remark}

\section{Proofs of Theorems \ref{T:Gromov}--\ref{T:fixedpoint2}}\label{S:proof2}

In this section we first state two special cases of Theorem
\ref{T:cen-rep} (Corollaries \ref{C:cen-rep1} and \ref{C:cen-rep2}
below), and then use them to prove Theorems
\ref{T:Gromov}--\ref{T:fixedpoint2}.

\begin{corollary}\label{C:cen-rep1}
Let $M$ be a connected analytic manifold with a rigid unimodular
analytic $\A$-structure of finite volume, let $G$ be a connected Lie
subgroup of $\Iso(M)$, and let $\cV\supset\cG$ be a $G$-invariant
subspace of $\Kill(M)$ such that $G$ is $\rho_\cV$-discompact. Then
there exist a $G$-invariant open dense subset $U\subset M$ and a
representation $\rho:\pi_1(M)\to\GL(d,\RR)$ such that
\begin{itemize}
\item[(1)] $T_{\tilde{x}}\widetilde{G}\tilde{x}\subset\ev_{\tilde{x}}(\Kill(\widetilde{M})^\cV)$ for every
$\tilde{x}\in\pi^{-1}(U)$, and
\item[(2)] for every $x\in
U$, we have $\Lg(x)\lhd\Lg$ and
$\ad(\Lg/\Lg(x))\prec\LL\left(\overline{\overline{\rho(\pi_1(M))}}\right)$.
\end{itemize}
\end{corollary}

\begin{proof}
The finite smooth measure $\mu$ induced by the unimodular structure
is preserved by $G$. By Theorem \ref{T:cen-rep}, the conclusions of
the corollary hold with $U$ replaced by a $G$-saturated
constructible set $M'\subset M_\reg$ with $\mu(M_\reg\setminus
M')=0$. Since $\mu$ is smooth, by Corollary \ref{C:constructible},
$M'$ contains a $G$-invariant open dense subset $U$ of $M$. This
completes the proof.
\end{proof}

\begin{remark}\label{R:cocompact}
Let $G$ be the discompact radical of $\Iso(M)$ with respect to the
natural representation of $\Iso(M)$ in $\Kill(M)$ (see Proposition
\ref{P:discompact-radical}). Then $G$ is a connected closed normal
subgroup of $\Iso(M)$ such that $\Iso(M)/G$ is locally isomorphic to
a compact Lie group, and the conditions of Corollary
\ref{C:cen-rep1} are satisfied for $\cV=\Kill(M)$. This justifies
Remark \ref{R:cen-rep}(2).
\end{remark}

\begin{corollary}\label{C:cen-rep2}
Let $M$ be a connected compact analytic manifold with a rigid
analytic $\A$-structure, and let $G$ be a connected split solvable
Lie group which acts analytically and isometrically on $M$. Then for
any $G$-minimal set $M_0\subset M$ the following assertions hold.
\begin{itemize}
\item[(1)] If $\cV\supset\cG$ is a $G$-invariant
subspace of $\Kill(M)$ such that $G$ is $\rho_\cV$-discompact, then
for any $\tilde{x}\in\pi^{-1}(M_0)$, we have
$T_{\tilde{x}}\widetilde{G}\tilde{x}\subset
\ev_{\tilde{x}}(\Kill(\widetilde{M})^\cV)$. In particular, we always
have $T_{\tilde{x}}\widetilde{G}\tilde{x}\subset
\ev_{\tilde{x}}(\Kill(\widetilde{M})^\cG)$ for any
$\tilde{x}\in\pi^{-1}(M_0)$.
\item[(2)] There exists a
representation $\rho:\pi_1(M)\to\GL(d,\RR)$ such that for any $x\in
M_0$, we have $\Lg(x)\lhd\Lg$ and
$\ad(\Lg/\Lg(x))\prec\LL\left(\overline{\overline{\rho(\pi_1(M))}}\right)$.
\end{itemize}
\end{corollary}

\begin{proof}
Since $G$ is solvable, it is amenable. So there exists a
$G$-invariant finite Borel measure $\mu$ on $M$ supported on $M_0$.
Since $M$ is analytic and compact, we have $M_\reg=M$. By Lemma
\ref{L:discompact}, $G$ is $\rho_\cG$-discompact. Thus Theorem
\ref{T:cen-rep} implies that the conclusions of the corollary hold
with $M_0$ replaced by a $G$-saturated constructible $\mu$-conull
subset $M'$ of $M$. By Lemma \ref{L:minimal-saturated}, we have
$M_0\subset M'$, and the corollary follows.
\end{proof}

\begin{remark}
By \cite[Lem. 2.6]{St}, $\Lg(x)$ is independent of the choice of
$x\in M_0$.
\end{remark}

Now we prove Theorems \ref{T:Gromov}--\ref{T:fixedpoint2}.

\begin{proof}[Proof of Theorem \ref{T:Gromov}]
We prove the theorem by showing that if $\Iso(M)_0$ is nontrivial
and has a discrete center, then it is compact and semisimple. Since $M$ is
compact, every Killing field on $M$ is complete. So the
infinitesimal action $\LL(\Iso(M))\to\Kill(M)$ induced by the
$\Iso(M)$-action is an $\Iso(M)$-equivariant isomorphism. Let $G$ be
the $\Ad$-discompact radical of $\Iso(M)$. Then $G$ is discompact
with respect to its natural representation in $\Kill(M)$, and
$\Iso(M)/G$ is locally isomorphic to a compact Lie group. Since $M$
is simply connected, the $\cV=\Kill(M)$ case of Corollary
\ref{C:cen-rep1}(1) implies that there exists an open dense subset
$U\subset M$ such that $T_xGx\subset\ev_x(Z(\Kill(M)))$ for every
$x\in U$. Since $\LL(\Iso(M))\cong\Kill(M)$ and $\Iso(M)_0$ has a
discrete center, we have $Z(\Kill(M))=0$. Thus for every $x\in U$ we
have $T_xGx=0$, i.e., $x$ is fixed by $G$. But $U$ is dense in $M$.
So $G$ acts trivially on $M$. Hence $G$ is trivial. Thus $\Iso(M)_0$
is locally isomorphic to a compact Lie group. But it has a discrete
center. So it is compact and semisimple.
\end{proof}

\begin{remark}\label{R:complete}
In the proof of Theorem \ref{T:Gromov}, the compactness of $M$ is
only used to ensure that every Killing field is complete. If the
geometric structure is a pseudo-Riemannian structure or a linear
connection plus a volume density, this remains true if $M$ is not
compact but geodesically complete (see \cite[Thm. VI.2.4]{KN}). This
justifies Remark \ref{R:Gromov}.
\end{remark}

The following lemma will be used in the proofs of Theorems
\ref{T:large} and \ref{T:fixedpoint2}.

\begin{lemma}\label{L:virtually}
Let $\Gamma$ be a group, and let $\rho:\Gamma\to\GL(d,\RR)$ be a
representation. If $\Gamma$ is virtually abelian (resp. virtually
nilpotent), then
$\LL\left(\overline{\overline{\rho(\Gamma)}}\right)$ is abelian
(resp. nilpotent).
\end{lemma}

\begin{proof}
Let $\Gamma_0$ be an abelian (resp. nilpotent) subgroup of $\Gamma$
of finite index. By \cite[p. 60, Cor.2]{Bo91},
$\LL\left(\overline{\overline{\rho(\Gamma_0)}}\right)$ is abelian
(resp. nilpotent). Let $\gamma_1,\ldots,\gamma_k\in\Gamma$ be such
that $\Gamma=\bigcup_{i=1}^k\gamma_i\Gamma_0$. Then
$\overline{\overline{\rho(\Gamma)}}=\bigcup_{i=1}^k\rho(\gamma_i)\overline{\overline{\rho(\Gamma_0)}}$.
Thus $\overline{\overline{\rho(\Gamma_0)}}$ is a closed subgroup of
finite index in $\overline{\overline{\rho(\Gamma)}}$. Hence
$\LL\left(\overline{\overline{\rho(\Gamma)}}\right)=\LL\left(\overline{\overline{\rho(\Gamma_0)}}\right)$
is abelian (resp. nilpotent).
\end{proof}

\begin{proof}[Proof of Theorem \ref{T:large}]
We may view $G$ as a connected Lie subgroup of $\Iso(M)$. Since $G$
is split solvable, by Lemma \ref{L:discompact}(1), it is
$\rho_\cG$-discompact. Then the $\cV=\cG$ case of Corollary
\ref{C:cen-rep1}(2) implies that there exist an open dense subset
$U\subset M$ and a representation $\rho:\Gamma\to\GL(d,\RR)$ such
that for every $x\in U$, we have $\Lg(x)\lhd\Lg$ and
$\ad(\Lg/\Lg(x))\prec\LL\left(\overline{\overline{\rho(\Gamma)}}\right)$,
where $\Gamma=\pi_1(M)$. Note that since $U$ is dense in $M$ and the
$G$-action on $M$ is faithful, we have $\bigcap_{x\in U}\Lg(x)=0$.
We prove the theorem by showing that if $\Gamma$ is finite (resp.
virtually abelian, virtually nilpotent), then $G$ is abelian (resp.
at most $2$-step nilpotent, nilpotent).

(1) If $\Gamma$ is finite, then
$\LL\left(\overline{\overline{\rho(\Gamma)}}\right)=0$. So for every
$x\in U$ we have $\ad(\Lg/\Lg(x))=0$, and hence $\Lg/\Lg(x)$ is
abelian. This implies that $[\Lg,\Lg]\subset\Lg(x)$ for every $x\in
U$. So $[\Lg,\Lg]\subset\bigcap_{x\in U}\Lg(x)=0$. Hence $G$ is
abelian.

(2) If $\Gamma$ is virtually abelian, then by Lemma
\ref{L:virtually},
$\LL\left(\overline{\overline{\rho(\Gamma)}}\right)$ is abelian. So
for every $x\in U$, $\ad(\Lg/\Lg(x))$ is abelian. This implies that
$\Lg/\Lg(x)$ is at most $2$-step nilpotent, and hence
$[\Lg,[\Lg,\Lg]]\subset\Lg(x)$ for every $x\in U$. Thus
$[\Lg,[\Lg,\Lg]]\subset\bigcap_{x\in U}\Lg(x)=0$. Hence $G$ is at
most $2$-step nilpotent.

(3) If $\Gamma$ is virtually nilpotent, then by Lemma
\ref{L:virtually},
$\LL\left(\overline{\overline{\rho(\Gamma)}}\right)$ is nilpotent.
So for every $x\in U$, $\ad(\Lg/\Lg(x))$, hence $\Lg/\Lg(x)$, is
nilpotent. This implies that $\Lg_\infty\subset\Lg(x)$ for every
$x\in U$, where $\Lg=\Lg_1\supset\Lg_2\supset\cdots$ is the lower
central series of $\Lg$ and $\Lg_\infty=\bigcap_{i=1}^\infty\Lg_i$.
Thus $\Lg_\infty\subset\bigcap_{x\in U}\Lg(x)=0$. Hence $G$ is
nilpotent.
\end{proof}

\begin{proof}[Proof of Theorem \ref{T:fixedpoint1}]
Let $M_0\subset M$ be a $G$-minimal set. Since $M$ is compact, we
have $\LL(\Iso(M))\cong\Kill(M)$, and the adjoint representation of
$\Iso(M)$ is equivalent to the natural representation of $\Iso(M)$
in $\Kill(M)$. Thus $G$ is $\rho_\cV$-discompact for $\cV=\Kill(M)$,
and we have $Z(\Kill(M))=0$. By Corollary \ref{C:cen-rep2}(1), for
every $x\in M_0$ we have $T_xGx\subset\ev_x(Z(\Kill(M)))=0$. So $x$
is fixed by $G$. Hence $G$ acts trivially on $M_0$. But $M_0$ is
$G$-minimal. So it consists of a single point.
\end{proof}

\begin{remark}\label{R:fixedpoint1}
If $H$ is a connected noncompact semisimple Lie subgroup of
$\Iso(M)$ with an Iwasawa decomposition $H=KAN$, then $G=AN$
satisfies condition (1) in Theorem \ref{T:fixedpoint1}. Indeed, if
we denote the adjoint representation of $\Iso(M)$ by $\rho$, then
$\rho(H)=\rho(K)\rho(A)\rho(N)$ is an Iwasawa decomposition of
$\rho(H)$. Thus
$\overline{\overline{\rho(G)}}=\overline{\overline{\rho(A)\rho(N)}}$
is an $\RR$-split solvable real algebraic group, and hence is
discompact by Corollary \ref{C:discompact}.
\end{remark}

\begin{proof}[Proof of Theorem \ref{T:fixedpoint2}]
Let $M_0\subset M$ be a $G$-minimal set. By Corollary
\ref{C:cen-rep2}(2), there exists a representation
$\rho:\Gamma\to\GL(d,\RR)$ such that for every $x\in M_0$, we have
$\Lg(x)\lhd\Lg$ and
$\ad(\Lg/\Lg(x))\prec\LL\left(\overline{\overline{\rho(\Gamma)}}\right)$,
where $\Gamma=\pi_1(M)$. We prove that under either condition of the
theorem, we have $[\Lg,\Lg]\subset\Lg(x)$ for every $x\in M_0$.

(1) Suppose that $\Gamma$ is finite. Then
$\LL\left(\overline{\overline{\rho(\Gamma)}}\right)=0$. This implies
that $\ad(\Lg/\Lg(x))=0$ for every $x\in M_0$. So $\Lg/\Lg(x)$ is
abelian, and hence $[\Lg,\Lg]\subset\Lg(x)$.

(2) Suppose that $\Gamma$ is virtually nilpotent and
$[\Lg,[\Lg,\Lg]]=[\Lg,\Lg]$. Let
$\Lg=\Lg_1\supset\Lg_2\supset\cdots$ be the lower central series of
$\Lg$, and let $\Lg_\infty=\bigcap_{i=1}^\infty\Lg_i$. Then
$\Lg_\infty=[\Lg,\Lg]$. By Lemma \ref{L:virtually},
$\LL\left(\overline{\overline{\rho(\Gamma)}}\right)$ is nilpotent.
So for every $x\in M_0$, $\ad(\Lg/\Lg(x))$, and hence $\Lg/\Lg(x)$,
is nilpotent. Thus $[\Lg,\Lg]=\Lg_\infty\subset\Lg(x)$.

Now let $x\in M_0$. Since $[\Lg,\Lg]\subset\Lg(x)$, we have
$(G,G)\subset G(x)$. This means that $x$ is fixed by $(G,G)$. So
$(G,G)$ acts trivially on $M_0$. This completes the proof.
\end{proof}

\section*{Acknowledgement} The author would like to thank David Fisher for his interest in this work. He also thank the referee for helpful comments.

\end{document}